\documentclass{article}

\usepackage{amssymb,amsmath,latexsym,amsthm}
\usepackage{mathtools}
\usepackage[english]{babel}
\usepackage{cleveref}
\usepackage{float}
\usepackage{tikz-cd}
\usepackage{mathrsfs}
\usepackage{enumitem}
\usepackage{adjustbox}
\usepackage[ruled,vlined,onelanguage]{algorithm2e}

\theoremstyle{definition}

\numberwithin{equation}{section}

\DeclareMathOperator{\Spec}{Spec}
\DeclareMathOperator{\Proj}{Proj}

\DeclareMathOperator{\Aut}{Aut}
\DeclareMathOperator{\Pic}{Pic}

\DeclareMathOperator{\res}{res}
\DeclareMathOperator{\Div}{Div}

\DeclareMathOperator{\RG}{R\Gamma}

\DeclareMathOperator{\HH}{H}
\DeclareMathOperator{\DD}{D}

\DeclareMathOperator{\id}{id}
\DeclareMathOperator{\ind}{ind}
\DeclareMathOperator{\Mod}{Mod}

\DeclareMathOperator{\GL}{GL}

\DeclareMathOperator{\Hom}{Hom}
\DeclareMathOperator{\Homcr}{Hom_{cr}}

\DeclareMathOperator{\RHom}{RHom}

\DeclareMathOperator{\Ab}{Ab}

\DeclareMathOperator{\Mat}{Mat}
\DeclareMathOperator{\Gal}{Gal}

\DeclareMathOperator{\Tot}{Tot}

\DeclareMathOperator{\Jac}{Jac}
\DeclareMathOperator{\Poly}{Poly}
\DeclareMathOperator{\cone}{Cone}

\newcommand{\pf}{{\rm pf}}
\newcommand{\sn}{{\rm sn}}

\newcommand{\R}{{\rm R}}

\newcommand{\FF}{\mathbb{F}}

\newcommand{\ZZ}{\mathbb{Z}}

\newcommand{\PP}{\mathbb{P}}
\newcommand{\bareta}{{\bar{\eta}}}

\newcommand{\LL}{\mathscr{L}}
\newcommand{\GG}{\mathbb{G}}

\newcommand{\F}{\mathcal{F}}
\newcommand{\GO}{\mathfrak{G}_0}

\newcommand{\tX}{{\tilde{X}}}

\newcommand{\tZ}{{\tilde{Z}}}

\newcommand{\Gk}{\Gal(k |k_0)}

\newcommand{\sep}{{\rm sep}}
\newcommand{\Ksep}{K^{\rm sep}}

\newcommand{\cont}{\mathrm{cont}}

\newcommand{\red}{{\rm red}}

\renewcommand{\div}{\mathrm{div}}

\newcommand{\nt}{^{\langle n\rangle}}
\newcommand{\lt}{^{\langle \ell\rangle}}

\newcommand{\Xnt}{{X\nt}}

\newcommand{\Unt}{{U\nt}}
\newcommand{\Vnt}{{V\nt}}
\newcommand{\Qlt}{_D^{\langle \ell\rangle}}

\newtheorem{df}{Definition}[section]
\newtheorem{rk}[df]{Remark}
\newtheorem{theorem}[df]{Theorem}
\newtheorem{prop}[df]{Proposition}
\newtheorem{lem}[df]{Lemma}

\newtheorem{cor}[df]{Corollary}
\newtheorem*{term}{Terminology}
\newtheorem*{ack}{Acknowledgements}

\begin{document}

\title{Computing the cohomology of constructible étale sheaves on curves}


\author{Christophe Levrat}



\maketitle


\begin{abstract}
We present an explicit expression for the cohomology complex of a constructible sheaf of abelian groups on the small étale site of an irreducible curve over an algebraically closed field, when the torsion of the sheaf is invertible in the field. This expression only involves finite groups, and is functorial in both the curve and the sheaf. In particular, we explain how to compute the Galois action on this complex. We also present an algorithm which computes this complex and study its complexity. We illustrate this algorithm with several examples.
\end{abstract}
\tableofcontents

\bigskip
\section{Introduction}
~

Let $X_0$ be an algebraic curve over a field $k$. Let $n$ be a positive integer invertible in $k_0$, and $\F_0$ be a constructible sheaf of $\ZZ/n\ZZ$-modules on the (small) étale site of $X_0$. Denote by $k$ a separable closure of $k_0$, by $X$ the base change of $X_0$ to $k$, and by $\F$ the restriction of $\F_0$ to $X$. The étale cohomology complex $\RG(X,\F)$ is equipped with an action of $\Gal(k|k_0)$. Given the curve $X$ and a suitable explicit description of the sheaf $\F$, we are interested in computing a finite extension $k_1$ of $k_0$ and a complex of $(\ZZ/n\ZZ)[\Gal(k_1|k_0)]$-modules which represents $\RG(X,\F)$. 

The computability of étale cohomology groups of torsion sheaves on schemes of finite type over algebraically closed fields was proved in 2014 by Poonen, Testa and van Luijk \cite[Th. 7.9]{poonen_testa} in characteristic zero, and by Madore and Orgogozo in arbitrary characteristic \cite[Th. 0.1]{mo}. However, the algorithms described in these articles are not efficient enough to be used in practice, and the only known result about their complexity is that Madore and Orgogozo's algorithm is primitive recursive \cite[Prop. 4.1.9]{mathese}. In the case of smooth curves, some more efficient algorithms are known.
When $X$ is a smooth projective curve, $\HH^1(X,\mu_n)$ is canonically isomorphic to the $n$-torsion of $\Pic(X)$; two algorithms, one developed by  Huang and Ierardi \cite{huang_counting}, the other by Couveignes \cite{couveignes_linearizing}, compute this group when the base field $k_0$ is finite. Jin's algorithm \cite{jinbi_jin} computes $\HH^1(X,\F)$ where $X$ is a smooth curve and $\F$ is locally constant. 

In this paper, we consider the case where $X_0$ is an integral curve over a field, and $\F_0$ is a constructible sheaf (or even a complex of such sheaves) on $X_0$. We present an explicit expression for the cohomology complex $\RG(X,\F)$ with the aforementioned Galois action, as well as an algorithm which computes this complex (under some classical computability assumptions on the field $k_0$). The latter makes use of an existing algorithm computing $\HH^1(X,\mu_n)$, such as those mentioned above. We also provide complexity bounds for this method. In particular, in the case of locally constant sheaves on smooth projective curves over finite fields, the complexity of computing $\HH^1(X,\F)$ using this algorithm is lower than Jin's.
In the case where the base field $k_0$ is finite, we present an idea that should allow us to reduce the complexity of this algorithm, and explain why this would be a crucial step towards developing a polynomial-time point counting algorithm for surfaces.

In section \ref{sec:cover_curves}, we investigate the properties of the minimal Galois cover of a scheme trivialising the $\ZZ/n\ZZ$-torsors on this scheme, as well as its construction in the case of curves. In section \ref{sec:cover}, we explain how to compute the cohomology of a locally constant sheaf on a scheme of cohomological dimension at most $1$. Section \ref{sec:explicit} contains the proof of the main theorem: an explicit expression of $\RG(X,-)$ when $X$ is a curve over a field of cohomological dimension at most $1$. We then present in section \ref{sec:algorithmic} the algorithms used to compute the cohomology of a constructible sheaf on such a curve, as well as their complexity. We also describe the potential application of our algorithms to point counting on surfaces over finite fields. In sections \ref{sec:examples1} and \ref{sec:examples2}, we illustrate these algorithms in two situations.

\begin{term}
Unless explicitly stated otherwise, the word \textit{cover} denotes a surjective finite étale map. A \textit{Galois cover} is always supposed to be connected. 
\end{term}

\section{The cover trivialising $\ZZ/n\ZZ$-torsors}\label{sec:cover_curves}

\subsection{General construction and properties}\label{subsec:1.1}

Let $n$ be a positive integer. We will denote by $\Lambda$ the ring $\ZZ/n\ZZ$. Given a (discrete) $\Lambda$-module $M$, we will denote by $M^\vee$ its $\Lambda$-dual.

\begin{lem}
\label{lem:sgcar} Let $G$ be a locally compact topological group. Consider the abelian group $\Lambda$, with the trivial action of $G$. Suppose the continuous cohomology group  $\HH^1(G,\Lambda)$ is finite. There is a unique closed normal subgroup $S$ of $G$ such that $G/S$ is isomorphic to the $\Lambda$-dual $\HH^1(G,\Lambda)^\vee$ of $\HH^1(G,\Lambda)$. Moreover, $S$ is a characteristic subgroup of $G$.
\begin{proof} Define $S$ as the closure in $G$ of $G^n[G,G]$. This group $S$ is a characteristic subgroup of $G$ (i.e. stable by all continuous automorphisms) because $G^n$ and $[G,G]$ are stable by automorphisms.
Since $\Lambda$ is an $n$-torsion abelian group, there is a canonical isomorphism \[ \Hom_\cont(G/S,\Lambda)\xrightarrow{\sim}\Hom_\cont(G,\Lambda)=\HH^1(G,\Lambda).\]
Pontryagin duality \cite[Th. 7.63]{hofmann_compact} applied to the locally compact abelian group $G/S$ yields the following isomorphism:
\[ \Hom_\cont(G/S,\Lambda)^\vee\xrightarrow{\sim}G/S.\]
Now let $H$ be a closed normal subgroup of $G$ such that $G/H$ is isomorphic to $\HH^1(X,\Lambda)^\vee$. It necessarily contains $S$ since $G/H$ is a $\Lambda$-module. 
Since $G/S$ and $G/H$ have the same finite cardinality, $H=S$.
\end{proof}
\end{lem}

\begin{cor}
Let $X$ be an integral noetherian scheme such that $\HH^1(X,\Lambda)$ is finite. Up to isomorphism, there is a unique étale Galois cover $\Xnt$ of $X$ with automorphism group isomorphic to $\HH^1(X,\Lambda)^\vee$.
\begin{proof}
This follows immediately from \Cref{lem:sgcar}, the canonical isomorphism \[\HH^1(\pi_1(X),\Lambda)\xrightarrow{\sim}\HH^1(X,\Lambda)\] and Grothendieck-Galois theory.
\end{proof}
\end{cor}
From now on, given such a scheme $X$, we will always denote by $\Xnt$ a Galois cover of $X$ as in the previous corollary.

\begin{prop}\label{prop:morphH1zero} 
Let $X$ be an integral noetherian scheme such that $\HH^1(X,\Lambda)$ is finite. For any finitely generated $\Lambda$-module $M$, the morphism $\HH^1(X,M)\to \HH^1(\Xnt,M)$ is trivial.
\begin{proof} Recall that $\Xnt$ corresponds to the open subgroup $S$ of $\pi_1(X)$, which is the closure of the subgroup $\pi_1(X)^n[\pi_1(X),\pi_1(X)]$. Since $M$ is $n$-torsion, any map $\pi_1(X)\to M$ is trivial on $S$, hence $\Hom(\pi_1(X),M)\to \Hom(\pi_1(\Xnt),M)$ is trivial.
\end{proof}
\end{prop}

\subsection{Explicit construction in the case of curves}

Let $U$ be a smooth integral curve over a field $k$. Denote by $K$ its function field. Let $n$ be a positive integer invertible in $k$. Recall the following description of $\HH^1(U,\mu_n)$ in terms of divisors on $U$.

\begin{lem}\label{lem:H1div} The group $\HH^1(U,\mu_n)$ is canonically isomorphic to the quotient of \[ \{ (D,f)\in \Div(U)\times K^\times\mid nD=\div(f)\} \]
by the subgroup of pairs $(\div(f),f^n)$ with $f\in K^\times$.
\begin{proof} This follows immediately from the corresponding description in terms of invertible sheaves given in \cite[040Q]{stacks}.
\end{proof}
\end{lem}
We will denote by $[D,f]$ the class of the pair $(D,f)$ in $\HH^1(U,\mu_n)$.
Now suppose that $k$ is separably closed. Let $X$ be the smooth compactification of $U$. Denote by $g$ the genus of $X$. Consider the closed complement $Z=\{ P_1,\dots,P_r\}$ of $U$ in $X$. 
When $r=0$ the lemma above is the usual description of $\HH^1(X,\mu_n)$ as the group of $n$-torsion points on the Jacobian of $X$.
When $r\geqslant 1$, the Gysin sequence \[ 0\to \HH^1(X,\mu_n)\to \HH^1(U,\mu_n)\to \HH^0(Z,\Lambda)\to \Lambda\to 0\]
shows that $\HH^1(U,\mu_n)$ is a free $\Lambda$-module of rank $2g-1+r$. Consider a basis $([D_i,g_i])_{1\leqslant i\leqslant s}$ of $\HH^1(U,\mu_n)$. For every integer $j\in \{1\dots s\}$, we denote by $V_j$ the normalisation of $U$ in $K(\sqrt[n]{g_1},\dots,\sqrt[n]{g_j})$, by $\phi_j$ the induced cover $U_j\to U$, and set $V=V_s$.

\begin{prop} The cover $\phi\colon V\to U$ is isomorphic to $\Unt\to U$.
\begin{proof} First of all, each cover $V_i\to V_{i-1}$ is étale, because it is constructed by taking an $n^{\text{th}}$ root of a function whose valuation at each point is a multiple of $n$. Let us check by induction on $j \in \{1,\dots,r\}$ that $V_j\to V_{j-1}$ is Galois with group $\mu_n(k)$. This is obviously true for $j=1$.
For any $i\in\{ 1,\dots,j-1\}$ the extension $V_i=V_{i-1}(\sqrt[n]{g_i})$ of $V_{i-1}$ is Galois with group $\mu_n(k)$ by the induction hypothesis. 
The Hochschild-Serre spectral sequence yields an exact sequence \[ 0 \to\HH^1(\mu_n(k),\Lambda) \to \HH^1(V_{i-1},\Lambda) \to \HH^1(V_i,\Lambda).\] 
Therefore, the kernel of $\phi_i^\star \colon \HH^1(U,\Lambda)\to \HH^1(V_i,\Lambda)$ is the direct sum $\Lambda [D_1,g_1]\oplus \dots \oplus\Lambda [D_i,g_i]\simeq \Lambda^i$, and $[D_j,g_j]$ is not in the kernel; the order of $\phi_i^\star [D_j,g_j]$ in $\HH^1(V_i,\mu_n)$ is still $n$. 
Thus, $V\to U$ is finite étale of order $n^{r}$. The field $k$ being separably closed, the extension $k(V)$ of $k(U)$ is the splitting field of the polynomials $T^n-g_1,\dots,T^n-g_r$, so it is Galois. The morphism $V\to U$ is therefore an étale Galois cover. An element of the group $\Aut(V|U)$ is an automorphism defined by $(\sqrt[n]{g_1}\mapsto \zeta_1\sqrt[n]{g_1},\dots, \sqrt[n]{g_{r}}\mapsto \zeta_{r}\sqrt[n]{g_{r}})$, where the $\zeta_i$ are $n^{\text{th}}$ roots of unity in $k$; the group $\Aut(V|U)$ is therefore canonically isomorphic to $\Hom_\Lambda(\HH^1(U,\mu_n),\mu_n)=\HH^1(U,\Lambda)^\vee$.  \Cref{lem:sgcar} now ensures that $V$ is isomorphic to $\Unt$.
\end{proof}
\end{prop}

\begin{rk} We will often consider the following situation. Let $V\to U$ be an étale Galois cover of smooth integral curves over $k$. The cover $\Vnt\to U$ is still Galois because $\Vnt\to V$ is characteristic (i.e. if $V\to S$ is a Galois cover, $V\nt\to S$ is Galois as well). Sometimes, we will need to compute a subcover $V'\to V$ of $\Vnt$ by taking $n^{\text{th}}$ roots of functions $g_1,\dots,g_t$ generating a submodule of $\HH^1(V,\mu_n)$. In that case, $V'\to U$ is still Galois if and only if the submodule generated by $g_1,\dots,g_t$ is stable under the action of $\Aut(V|U)$. This is the case for instance if we take the submodule of elements of $\HH^1(V,\mu_n)$ defined over some subfield of $k$ over which the elements of $\Aut(V|U)$ are also defined.
\end{rk}

\subsection{Ramification at infinity}

Let $U$ be an integral affine curve over an algebraically closed field $k$. Denote by $X$ the smooth compactification of $U$. Let $n$ be an integer invertible in $k$, and denote by $\Lambda$ the ring $\ZZ/n\ZZ$. Let us study the ramification at infinity of $\Unt\to U$, i.e. the ramification of the smooth compactification $X'$ of $\Unt$ above the points $P_0,\dots,P_r$ of $X-U$. The map $\HH^2_Z(X,\mu_n)\to \HH^2(X,\mu_n)$ can be expressed, using the Gysin isomorphism $\HH^0(Z,\Lambda)\to \HH^2_Z(X,\mu_n)$ and the isomorphism $\HH^2(X,\mu_n)\to \Lambda$, as the sum map $\HH^0(Z,\Lambda)\to \Lambda$. A basis of its kernel is given by $(P_1-P_0,\dots,P_r-P_0)$. Consider functions $g_1,\dots,g_r\in k(X)$ such that \[\div(g_i)=nD_i+(P_i-P_0)\] where $D_i\in\Div^0(X-\{P_0,\dots,P_r\})$. The cover $X'\to\Xnt$ corresponds to the function field extension \[k(\Xnt)(\sqrt[n]{g_1},\dots,\sqrt[n]{g_r})\]
of $k(\Xnt)$. Let $i\in \{1\dots r\}$. The extension $k(\Xnt)(\sqrt[n]{g_j},j\neq i)$ of $k(\Xnt)$ yields a cover $Y_i\to \Xnt$ which is unramified above $P_i$ because $v_{P_i}(g_j)=0$. The cover $X'\to Y_i$, however, is ramified at $P_i$; let $Q_i$ be a preimage of $P_i$ in $Y_i$. Since $v_{Q_i}(g_i)=1$, the fibre $X'_{Q_i}$ is isomorphic to $k[x]/(x^n)$, and the ramification index of $X'\to Y_i$ at any preimage of $Q_i$ is $n$. Above $P_i$, there are exactly $\frac{1}{n}|\HH^1(U,\Lambda)|$ points of $X'$, each of with ramification index $n$. Let $R_i$ be a preimage of $Q_i$ in $X'$.
The inertia subgroup $I_{R_{i}|P_i}\subset \Aut(\Unt|U)$ of $R_i$ fits in the short exact sequence
\[ 0\to I_{R_{i}|Q_i} \to I_{R_{i}|P_i} \to I_{Q_i|P_i} \to 0\]
(see \cite[0BU7]{stacks}). As $I_{Q_i|P_i}=0$, there are isomorphisms \[I_{R_i|P_i}=I_{R_i|Q_i}=\Aut(X'|Y_i)\simeq\Lambda.\]
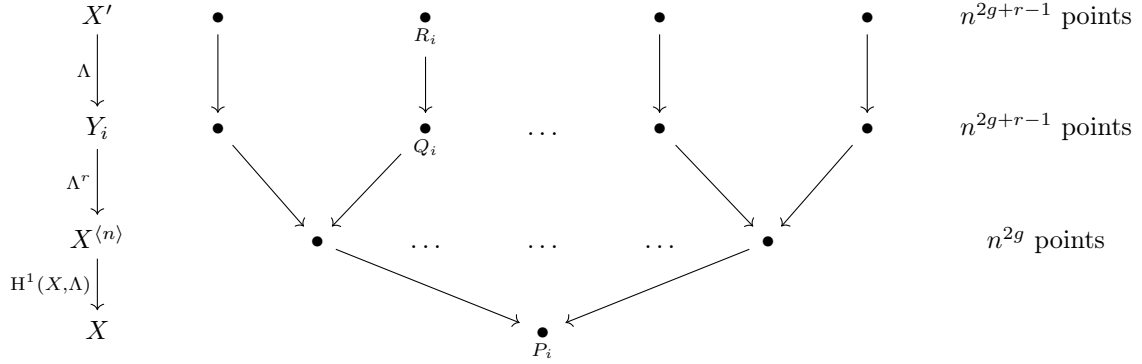
\begin{figure}[H]\label{fig:ramification}
\[\adjustbox{max width=\textwidth}{$
\begin{tikzcd}
X' \arrow[d,"\Lambda",swap] & \bullet \arrow[d]& &\underset{R_i}\bullet\arrow[d] & & \bullet\arrow[d] & & \bullet\arrow[d] & n^{2g+r-1}\text{ points} \\
Y_i\arrow[d,"\Lambda^r",swap] & \bullet\arrow[dr]& &\underset{Q_i}\bullet\arrow[dl] & \dots & \bullet\arrow[dr] & & \bullet\arrow[dl] & n^{2g+r-1}\text{ points}\\
\Xnt\arrow[d,"{\HH^1(X,\Lambda)}",swap] & &  \bullet\arrow[drr] & \dots & \dots & \dots & \bullet\arrow[dll] & &n^{2g}\text{ points}  \\
X & & & & \underset{P_i}{\bullet}  & & & & 
\end{tikzcd}
$}\]
\caption{Ramification at infinity of the cover $\Unt\to U$}
\end{figure}
As a subgroup of $\Aut(\Unt|U)$, the group $I_{R_i|P_i}$ is generated by $\sqrt[n]{g_i}\mapsto \zeta\sqrt[n]{g_i}$, where $\zeta$ is a primitive $n^{\text{th}}$ root of unity in $k$. The results regarding preimages and ramification indices also apply to $P_0$, which had been chosen arbitrarily. With the above notation, the subgroup $I_{R_0|P_0}$ is generated by the automorphism \[(\sqrt[n]{g_1},\dots,\sqrt[n]{g_r})\mapsto (\zeta\sqrt[n]{g_1},\dots,\zeta\sqrt[n]{g_r}).\]

\subsection{Cohomology of the ramification groups}\label{subsec:ramgrp}
Let $n$ be a positive integer. Denote by $\Lambda$ the ring $\ZZ/n\ZZ$.
Let us consider an étale Galois cover $V\to U$ of smooth integral curves over an algebraically closed field $k$ of characteristic prime to $n$. Let $K$ be the function field of $U$. Denote by $X$ (resp. $Y$) the smooth compactification of $U$ (resp. $V$). Set $G=\Aut(V|U)$. 
Let $x$ be a closed point of $X-U$. Consider an inertia group $I\subset\Gal(\Ksep|K)$ at $x$, and one of its finite quotients $I_y\subset G$, which is the stabiliser in $G$ of a closed point $y\in Y-V$ mapping to $x$. Denote by $P\triangleleft I$ and $P_y\triangleleft I_y$ the wild inertia subgroups. There are canonical isomorphisms \cite[09EE]{stacks}
\[  I/P\xrightarrow{\sim}\lim_{m\nmid p}\mu_m(k)~~\text{and}~~ I_y/P_y\xrightarrow{\sim}\mu_e(k)\] where $e$ is the ramification index of $Y\to X$ at $y$. From now on, we assume that $n$ divides $e$. Let $M$ be a $\Lambda[G]$-module.



\begin{prop}\label{cor:cohI} In the situation described above, the canonical maps
\[ \tau_{\leqslant 1}\RG(I_y/P_y,M^{P_y})\to\RG(I/P,M^P)\to\RG(I,M)
\]
are quasi-isomorphisms. 
\begin{proof}
Let $\sigma$ denote a pro-generator of $I/P$, and $\sigma_y$ its image in $I_y/P_y$. The actions of $\sigma$ and $\sigma_y$ on $M$ are equal. Set $N_y=\sum_{i=1}^e\sigma_y^i$. The usual results on the cohomology of (pro)-cyclic groups \cite[Calc. 6.2.1]{weibel} \cite[Prop. 8.1.4]{fulei} show that the map on the left hand side is: 
\[
\begin{tikzcd}
\RG(I/P,M^P) &\arrow[r,equal]&~  &M^P \arrow[r,"\sigma-\id"] & M^P \arrow[r] & 0 \\
\tau_{\leqslant 1}\RG(I_y/P_y,M^{P_y})\arrow[u]&\arrow[r,equal]&~  &M^P \arrow[r,"\sigma_y-\id"]\arrow[u,equal] & \ker(N_y)\arrow[r] \arrow[u,hookrightarrow] & 0\arrow[u]
\end{tikzcd}
\]
Since $n$ divides $e$, the action of $N_y$ on $M$ is trivial and $\ker(N_y)=M^P$, which shows that the map above is a quasi-isomorphism. The fact that the map on the right hand side is a quasi-isomorphism is shown in \cite[Prop. 8.1.4]{fulei} as well.
\end{proof}
\end{prop}

The following lemma shows how to explicitly construct the inverse of the map on the left hand side as a morphism of complexes; it will be used in \Cref{th:RGamma}. We denote by $\Homcr$ the groups of crossed homomorphisms: given a group $G$ and a $G$-module $M$, $\Homcr(G,M)$ is the group of maps $f\colon G\to M$ such that for all $g,h\in G$, $f(gh)=f(g)+g\cdot f(h)$.

\begin{lem}\label{lem:secIP} The canonical map $\Homcr(I_y/P_y,M^{P_y})\to \Homcr(I_y,M)$ has a section. 
\begin{proof} Let $u\colon I_y\to M$ be a crossed homomorphism.
Consider the commutative diagram\[
\begin{tikzcd}
I_y \arrow[d]\arrow[r,"u"] & M \arrow[r,"q"] & M_{P_y} \\
I_y/P_y \arrow[r,dashed] & M^{P_y}\arrow[ur,"\alpha"]\arrow[u] & 
\end{tikzcd}
\]
and set $f=q\circ u$. For all $x\in P_y$ and $g\in I_y$, the definition of $M_{P_y}$ ensures that $f(x g)=f(x)+q(x\cdot u(g))=f(x)+f(g)$. Hence, for any $x\in P_y$, $f(x^{|P_y|})=|P_y|f(x)$ is zero ; since multiplication by $|P_y|$ is an automorphism of $M$, this means that $f(x)=0$. Therefore, there is a quotient map $\bar f\colon I_y/P_y\to M_{P_y}$. Set $\bar u=\alpha^{-1}\circ \bar{f} \colon I_y/P_y\to M^{P_y}$. The map $u\mapsto\bar u$ is clearly linear. Moreover, $\bar u$ is still a crossed homomorphism since $\bar{u}(\bar g_1\bar g_2)=\alpha^{-1}f(g_1g_2)=\alpha^{-1}f(g_1)+q(g_1\cdot \alpha^{-1}u(g_2))$. The $I_y$-linearity of $\alpha^{-1}q$ concludes.
\end{proof}
\end{lem}

\subsection{A similar cover with Galois action}\label{subsec:X2gal}

Let $k_0$ be a perfect field, and $k$ be an algebraic closure of $k_0$. Let $n$ be an integer invertible in $k$. Denote by $\Lambda$ the ring $\ZZ/n\ZZ$. Let $V_0$ be a geometrically integral smooth curve over $k_0$. As usual, the base change $-\times_{k_0}k$ will be denoted by removing the subscript $_0$. We wish to compute a (connected) characteristic cover $V'_0$ of $V_0$ such that the map $\HH^1(V,\Lambda)\to \HH^1(V',\Lambda)$ is trivial. In the case where $V$ is connected and the elements of $\HH^1(V,\mu_n)$ are defined over $k_0$, the cover $\Vnt$ constructed in the previous sections comes from $k_0$, and we are done. However, this is not the case in general. The construction below is a refinement of the simple idea of computing the Galois orbit of $\Vnt$.

\paragraph{Construction of $V'_0$} The function field of $Y_0$ is of the form $k_0(x)[y]/(f)$, with $f\in k(x)[y]$. Denote by $k_1$ the algebraic closure of $k_0$ in $k_0(x)[y]/(f)$; since $k_0$ is perfect, $k_1$ is a separable extension. Let $V_1$ be a connected component of $V$. Let $([D_1,g_1],\dots,[D_r,g_r])$ be a basis of $\HH^1(V_1,\mu_n)$ as in \Cref{lem:H1div}. Denote by $k_2$ the minimal extension of $k_0$ over which the $g_i$ are defined. Let $L$ be the Galois closure of the extension of $k_0$ generated by $k_1,k_2$ and $\mu_n(k)$. Let $\alpha$ be a primitive element of the extension $L|k_1$, and $m\in k_1[t]$ its minimal polynomial. For $i\in\{1\dots r\}$, write $g_i=g_i'(\alpha,x,y)$ with $g'_i\in k_0(x)[t,y]/(m(t),f(x,y))$. Let $V'_0\to V_0$ be the normalisation of $V_0$ in the function field $k_0(\alpha,x,y)(\sqrt[n]{g'_1},\dots,\sqrt[n]{g'_r})$. The curve $V'$ is isomorphic to the $\Gal(L|k_0)$-orbit of $V_1\nt$; it has $[L:k_0]$ connected components, and the map $\HH^1(V,\Lambda)\to \HH^1(V',\Lambda)$ is trivial.

Note that since the degree of $V'_0\to (V_0\times_{k_1}L)$ is $n^r$ and that of $V_0\times_{k_1}L\to V_0$ is $[L:k_1]$, the degree of $V'_0\to V_0$ is $n^r[L:k_1]$.

\begin{prop}The cover $V'_0\to V_0$ is characteristic, i.e. if $V_0$ is an étale Galois cover of a curve $U_0$, then $V_0'\to U_0$ is still Galois.
\begin{proof}
Consider the situation where $V_0$ is an étale Galois cover of a smooth integral $k_0$-curve $U_0$. Here is how to explicitly compute the automorphism group of $V_0\to U_0$. 
The map $V'_0\to V_0$ has degree $n^r[L:k_1]$. Using the notations above, set $z_i=\sqrt[n]{g'_i}$. The elements of $\Aut(k_0(V'_0)|k_0(V_0))$ are defined by $t\mapsto \sigma(t),z_i\mapsto \zeta_i z_i$ where $\sigma\in \Gal(L|k_1)$ and $\zeta_i\in\mu_n(L)$. There are $\deg(V'_0\to V_0)=n^{r}[L:k_1]$ such automorphisms since $\mu_n(k)\subset L$, hence the cover $V'_{0}\to V_0$ is Galois.  
Let us now compute the automorphism group of $V'_0\to U_0$. Its elements are defined by\[ (t,x,y,z_1,\dots,z_r)\mapsto( \sigma(t),x',y',z'_1,\dots,z'_r) \]
where $\sigma\in \Gal(L|k_0)$, the pair $(x',y')$ is the image of $(x,y)$ under a $U_0$-automorphism of $V_0$ whose image in $\Gal(k_1|k_0)$ is the same as that of $\sigma$, and the elements $z'_i\in k_0(V'_0)$ satisfy ${z'_i}^n=\phi(g'_i)$. 
As expected, there are $\deg(V'_0|U_0)=n^r[L:k_1]\deg(V_0\to U_0)$ elements in $\Aut(V'_0| U_0)$, and $V'_0\to U_0$ is Galois. 
\end{proof}
\end{prop}

\begin{rk}\label{rk:galext}
If need be, the Galois extension $L$ of $k_0$ may be chosen to be a little larger, for instance to make sure that the points at infinity of the curve $V'$ are defined over $L$. This does not affect any of the previous proofs.
\end{rk}

\section{Schemes of cohomological dimension 1}\label{sec:cover}
\label{subsec:cd1}

Let $n$ be a positive integer, and $\Lambda=\ZZ/n\ZZ$. Let $X$ be an integral noetherian scheme, and $\bareta$ be a generic geometric point of $X$.

\begin{prop} Let $\LL$ be a finite locally constant sheaf of $\ZZ/n\ZZ$-modules on $X$. Let $Y\to X$ be an étale Galois cover such that $\LL|_Y$ is constant and the morphism 
\[ \HH^1(X,\LL)\to \HH^1(Y,\LL|_Y)\]
is trivial. Then the morphism
\[ \tau_{\leqslant 1}\RG(\Aut(Y|X),\LL_{\bareta})\to \tau_{\leqslant 1}\RG(X,\LL) \]
is a quasi-isomorphism.
\begin{proof} Let $G$ be the automorphism group of $Y\to X$. The associated Hochschild-Serre spectral sequence yields the following short exact sequence: 
\[ 0\to \HH^1(G,\LL_{\bareta})\to \HH^1(X,\LL)\to \HH^0(G,\HH^1(Y,\LL|_Y)) \] 
Hence the map $\HH^1(G,\LL_{\bareta})\to \HH^1(X,\LL)$ is an isomorphism, and \[\RG(G,\LL_{\bareta})\to \RG(X,\LL) \] yields isomorphisms on cohomology groups in degree 0 and 1. 
\end{proof}
\end{prop}

\begin{rk} If in addition $X$ has cohomological dimension 1 then \[ \tau_{\leqslant 1}\RG(\Aut(Y|X),\LL_{\bareta})\to \RG(X,\LL) \]
is a quasi-isomorphism.
\end{rk}

\begin{rk}
Here is how to construct a cover $Y$ as in the proposition, provided that $\HH^1(W,\Lambda)$ is finite. Pick an étale Galois cover $W\to X$ such that $\LL|_{W}$ is a constant sheaf. Set $Y=W\nt$: since $Y\to W$ is characteristic, the cover $Y\to X$ is still Galois, and \Cref{prop:morphH1zero} ensures that $\HH^1(X,\LL)\to\HH^1(Y,\LL|_Y)$ is trivial.
\end{rk}

Given a profinite group $H$, we will denote by $P_H(\Lambda)$ the usual projective resolution (sometimes called \textit{bar resolution}) of the trivial $\Lambda[[H]]$-module $\Lambda$. 

\begin{prop}\label{prop:H1cd1} Suppose $X$ is of cohomological dimension 1. Let $\LL=[\LL^0\to\LL^1\to\dots\to\LL^s]$ be a complex of finite locally constant sheaves of $\Lambda$-modules on $X$, and $Y$ be an étale Galois cover of $X$ such that each map $\HH^1(X,\LL^i)\to \HH^1(Y,\LL^i|_Y)$ is trivial. Write $G=\Aut(Y|X)$. Consider the double complex $B^{\bullet,\bullet}$ defined by $B^{i,j}=\Hom_{\Lambda[G]}(\tau_{\geqslant -1}P_{G}^{-j}(\Lambda),\LL_\bareta^i)$. Then $\RG(X,\LL)$ is represented by the total complex $\Tot B^{\bullet,\bullet}$.
\begin{proof} We wish to compute $\RG(X,\LL)=\RG(\pi_1(X),\LL_\bareta)=\RHom_{\Lambda[[\pi_1(X)]]}(\Lambda,\LL_\bareta)$, which is represented by the complex \[ \Hom^\bullet_{\Lambda[[\pi_1(X)]]}(P_\pi(\Lambda),\LL_\bareta)=\Tot(A^{\bullet,\bullet}) \] where $A^{i,j}=\Hom_{\Lambda[[\pi_1(X)]]}(P_{\pi_1(X)}^{-j}(\Lambda),\LL_\bareta^i)$.
The map $B^{\bullet,\bullet}\to A^{\bullet,\bullet}$ induced by the quotient map $\pi_1(X)\to G$ defines a morphism between the spectral sequence associated to these quotients.
Recall that for each $i\in\{1\dots s\}$, the map \[\tau_{\leqslant 1}\Hom_{\Lambda[G]}(P_G(\Lambda),\LL^i_\bareta)\to 
\Hom_{\Lambda[[\pi_1(X)]]}(P_{\pi_1(X)}(\Lambda),\LL^i_\bareta)\] is a quasi-isomorphism. 
Since the functor $\Hom(-,\LL^i_\bareta)$ is left exact, \[\Hom_{\Lambda[G]}(\tau_{\geqslant -1}P_G(\Lambda),\LL^i_\bareta)=\tau_{\leqslant 1}\Hom_{\Lambda[H]}(P_G(\Lambda),\LL^i_\bareta).\]
Hence the map $B^{\bullet,\bullet}\to A^{\bullet,\bullet}$ defines, on each column, a quasi-isomorphism of complexes. It therefore induces an isomorphism between the first pages of the corresponding spectral sequences (for the upward orientation) associated to $B^{\bullet,\bullet}$ and $A^{\bullet,\bullet}$: in position $(i,j)$, it is the isomorphism $\HH^j(H,\LL_\bareta^i)\to \HH^j(\pi_1(X),\LL_\bareta^i)$ for $j\leqslant 1$, and $0\to \HH^j(\pi_1(X),\LL_\bareta^i)=0$ otherwise. Therefore, the map between the abutments of these two spectral sequences is an isomorphism, i.e. the map $\Tot(B^{\bullet,\bullet})\to \Tot(A^{\bullet,\bullet})$ is a quasi-isomorphism.
\end{proof}
\end{prop}

\section{Explicit computation of $\RG$ of a (possibly singular) curve}\label{sec:explicit}

In this section, we are going to describe how to compute the cohomology of a complex of constructible sheaves on a curve. Let $k_0$ be a field, and $k$ be a separable closure of $k_0$. Consider a geometrically irreducible curve $X_0$ over $k_0$, and its base change $X$ over $k$. We are allowed to make the following additional assumptions, which do not alter the computed cohomology complex. \begin{itemize}
\item The field $k_0$ is perfect, hence $k$ is algebraically closed: the perfect closure $k_0^\pf$ of $k_0$  being a purely inseparable extension, the base change $-\times_{k_0} k_0^\pf$ induces an isomorphism on cohomology.
\item The curve $X$ is reduced: being a universal homeomorphism, the map $X_\red\to X$ induces an isomorphism on cohomology.
\item The curve $X$ has at worst multicross singularities \cite[0C1P]{stacks}: the seminormalisation map $X^\sn\to X$ being a universal homeomorphism \cite[0EUS]{stacks}, it induces an isomorphism on cohomology, so we may assume $X$ is seminormal. Since a seminormal curve over an algebraically closed field has at worst multicross singularities \cite[§2, Cor. 1]{davis_seminorm}, we may assume this is the case for $X$. 
\end{itemize}

\subsection{Cohomology with support in a 0-dimensional subscheme}

Let $k$ be an algebraically closed field. Let $n$ be an integer invertible in $k$, and $\Lambda=\ZZ/n\ZZ$. Let $X$ be an integral curve over $k$. Consider a nonempty closed zero-dimensional subscheme $i\colon Z\to X$, and its open complement $j\colon U\to X$.  Let $\RG_Z(X,-)$ denote cohomology with support in $Z$. Since $\RG_Z(X,-)=\oplus_{z\in Z}\RG_z(X,-)$, we consider a single point $z\in Z$ and focus on computing $\RG_z(X,-)$.

\begin{lem}\label{lem:cohsuppz} Let $\LL$ be a finite locally constant sheaf on $U$. \begin{itemize}[label=$\bullet$]
\item $\HH^0_z(X,j_!\LL)=\HH^0_z(X,j_\star\LL)=\HH^1_z(X,j_\star\LL)=0$
\item $\HH^1_z(X,j_!\LL)=\HH^0(z,i^\star j_\star\LL)$
\item $\HH^2_z(X,j_!\LL)=\HH^2_z(X,j_\star\LL)$
\item For all $i\geqslant 3$, $\HH^i_z(X,j_!\LL)=\HH^i_z(X,j_\star\LL)=0$.
\end{itemize}
\begin{proof} Since $\HH^0(X,j_\star\LL)\to \HH^0(X-z,j_\star\LL)$ is an isomorphism and $\HH^1(X,j_\star\LL)\to \HH^1(X-z,j_\star\LL)$ is always a monomorphism, the long exact sequence for cohomology with support associated to $j_\star\LL$ shows that $\HH^0_z(X,j_\star\LL)=\HH^1_z(X,j_\star\LL)=0$. 
Moreover, for any $j\geqslant 3$, the groups $\HH^{j-1}(X-z,j_\star\LL)$ and $\HH^j(X,\LL)$ are trivial, hence $\HH^j_z(X,j_\star\LL)=0$. 
Recall that $\HH^i_z(X,i_\star-)=\HH^i(z,-)$. The long exact sequence of $\HH^i_z(X,-)$ associated to the short exact sequence \[ 0\to j_!\LL\to j_\star\LL\to i_\star i^\star j_\star\LL\to 0\]
shows that $\HH^0_z(X,j_!\LL)=0$, $\HH^1_z(X,j_!\LL)=\HH^0(z,i^\star j_\star\LL)$ and also $\HH^2_z(X,j_!\LL)=\HH^2_z(X,j_\star\LL)$. 
The groups $\HH^i_z(X,j_!\LL)$ are also trivial as soon as $i\geqslant 3$.
\end{proof}
\end{lem}

Let us now compute the group $\HH^2_z(X,j_!\LL)=\HH^2_z(X,j_\star\LL)$. From now on, we assume $X$ to have only multicross singularities. Let $z_1,\dots,z_r$ be the preimages of $z$ in the normalisation $\tX$ of $X$. Denote by $X_{z}$ the strict henselisation of $X$ at $z$. It contains one closed point $z'$, as well as $r$ minimal primes $z'_1,\dots,z'_r$. Set $U_{z}\coloneqq U\times_X X_{z}$. Consider the following cartesian diagram.
\[
\begin{tikzcd}
U_{z}\arrow[d,"g'"]\arrow[r,"j'"] & X_{z}\arrow[d,"g"] &\arrow[l,"i'",swap] z' \arrow[d] \\
U \arrow[r,"j"] & X & \arrow[l,"i",swap] z
\end{tikzcd}
\]
The following proof is given in \cite[II, Prop. 1.1]{milneADT} in the special case of smooth curves.
\begin{lem}\label{lem:cohsuppzero} Let $\F$ be a sheaf on $U_{z}$. For every nonnegative integer $q$, the group $\HH^q(X_{z},j'_!\F)$ is trivial.
\begin{proof} The assertion holds for $q=0$ since $\HH^0(X_{z},j'_!\F)$ is the kernel of the map $\HH^0(X_{z},j'_\star\F)\to\HH^0(X_{z},i'_\star i'^\star j_\star\F)$, which is simply the identity map of $\HH^0(U_{z},\F)$. Let us first show that for any injective sheaf $J$ on $U_{z}$, the sheaf $j'_!\F$ on $X$ is acyclic. To this end, we are going to prove that
\[ 0\to j'_!J\to j'_\star J\to i'_\star i'^\star j'_\star J\to 0\] is an injective resolution of $j'_!J$ ;
the long exact sequence associated to this short exact sequence then shows that $\HH^q(X_{z},j'_\star J)=0$ for any $q\geqslant 1$.
Fix separable closures of $k(z'_1),\dots,k(z'_r)$, and denote by $I_1,\dots,I_r$ the associated Galois groups. The functor $i^\star j_\star$ may be rewritten as follows: \[ \begin{array}{rcl} \Mod_{I_1}\times\dots\times\Mod_{I_r}&\to &\Ab \\ (M_1,\dots,M_r)&\mapsto& M_1^{I_1}\times \dots\times M_r^{I_r}\end{array}\]
This functor admits a left adjoint, which sends an abelian group $M$ to $(M,\dots,M)$ with trivial $(I_1,\dots,I_r)$-action. Since this left adjoint is exact, $i'^\star j'_\star$ sends injectives to injectives. The functors $i'_\star$ and $j'_\star$ also send injectives to injectives, therefore the short exact sequence above is an injective resolution of $j'_!J$.
We are now ready to prove the result. Let $\F$ be a sheaf on $U_{z}$, and let $J^\bullet$ be an injective resolution of $\F$. Then $j'_!J^\bullet$ is an acyclic resolution of $j'_!\F$, and $\HH^q(X_{z},j'_!\F)$ is the $q^{\text{th}}$ cohomology group of the complex $\Gamma(X_{z},j'_!J^\bullet)$. The latter group is the image of $\F$ under the $q^{\text{th}}$ right derived functor of $\Gamma(X_{z},j'_!-)$, which is zero as proven above. 
\end{proof}
\end{lem}

Let $\nu\colon\tX\to X$ be the normalisation map. Set $\tilde z\coloneqq z\times_X\tX$. 

\begin{prop}\label{prop:RGznorm} Let $\LL$ be a finite locally constant sheaf on $U$. The map \[\RG_z(X,j_\star\LL)\to\RG_{\tilde z}(\tX,\nu^\star j_\star\LL)\] is a quasi-isomorphism.
\begin{proof} Denote by $\eta_1,\dots,\eta_r$ the generic points of the strict henselisations of $\tX$ at the preimages $z_1,\dots,z_r$ of $z$ in $\tX$. Denote by $\widetilde{X_{z}}$ the normalisation of the strict henselisation of $X$ at $z$. The normalisation of $U_{z}$ is \[ U_{z}\times_{X_{z}}\widetilde{X_{z}}=U\times_X(\tX_{z_1}\sqcup\dots\sqcup \tX_{z_r})=\eta_1\sqcup\dots\sqcup\eta_r\] and since normalisation is birational, the map \[ \eta_1\sqcup\dots\sqcup \eta_r\to U_{z} \]
is an isomorphism. \Cref{lem:cohsuppz} shows that it suffices to prove that \[ \HH^2_z(X,j_\star\LL)\to\HH^2_{\tilde z}(\tX,\nu^\star j_\star\LL)=\HH^2_{z_1}(\tX,\nu^\star j_\star\LL)\times\dots\times\HH^2_{z_r}(\tX,\nu^\star j_\star\LL)\] is an isomorphism. 
Let us compute $\HH^2_z(X,j_\star\LL)$. By excision, there is a canonical isomorphism \[ \HH^2_z(X,j_\star\LL)\xrightarrow{\sim} \HH^2_{z'}(X_{z},g^\star j_!\LL)=\HH^2_{z'}(X_{z},j'_!g'^\star \LL).\]
The long exact sequence in cohomology with support on $z'$ for the sheaf $j'_!g'^\star\LL$ reads:
\[ \HH^1(X_{z},j'_!g'^\star \LL)\to \HH^1(U_{z},g'^\star\LL)\to \HH^2_{z'}(X_{z},j'_!g'^\star\LL)\to \HH^2(X_{z},j'_!g'^\star\LL).\] \Cref{lem:cohsuppzero} ensures that the map \[\HH^1(U_{z},g'^\star\LL)\to \HH^2_{z'}(X_{z},j'_!g'^\star\LL)\]
is an isomorphism. Since the scheme $U_{z}$ is the coproduct of the points $z'_1,\dots,z'_r$, the group $\HH^1(U_{z},g'^\star\LL)$ is simply $\HH^1(z'_1,\LL_{z'_1})\times\dots\times \HH^1(z'_r,\LL_{z'_r})$. On the other hand, for $t\in \{1\dots r\}$, $\HH^2_{z_t}(\tX,\nu^\star j_\star\LL)$ is isomorphic to $\HH^1(\eta_t,\LL_{\eta_t})$, and the map 
\[ \HH^1(z'_t,\LL_{z'_t})\to \HH^1(\eta_t,\LL_{\eta_t})\]
is simply the isomorphism induced by $\eta_t\xrightarrow{\sim}z'_t$.
\end{proof}
\end{prop}
Denote by $K$ the function field of $X$, and by $K^\sep$ a separable closure of $K$.
\begin{cor}\label{cor:RGZ} For each point $z\in Z\times_X\tX$, choose a place of $\Ksep$ above $z$ and denote by $I_z$ the corresponding inertia group. The one-term complex \[ 0\to 0\to \bigoplus_{z\in Z\times_X\tX}\HH^1(I_z,M) \to 0\to \cdots\]
represents $\RG_Z(X,j_\star\LL)$.
\begin{proof}
This is merely rephrasing \Cref{prop:RGznorm} using the fact that only $\HH^2_z$ is nonzero, which was proven in \Cref{lem:cohsuppz}.
\end{proof}
\end{cor}

\subsection{Cohomology of constructible sheaves}

Let $k$ be an algebraically closed field. Let $n$ be an integer invertible in $k$. Denote by $\Lambda$ the ring $\ZZ/n\ZZ$. Let $X$ be an integral curve with multicross singularities over $k$. Denote by $\nu\colon\tilde X\to X$ its normalisation. Let $\F^\bullet=[\F^0\to\cdots\to \F^t]$ be a complex of constructible sheaves of $\Lambda$-modules on $X$. This section aims to give an explicit description of the cohomology complex $\RG(X,\F)$. \\

Let $j\colon U\to X$ be the inclusion of a regular open affine subscheme of $X$ on which every sheaf $\F^s$ is locally constant. Let $i\colon Z\to X$ be the inclusion of its closed complement with the reduced subscheme structure. Set $\LL^\bullet\coloneqq j^\star\F^\bullet$, and $M^\bullet\coloneqq \LL^\bullet_\bareta$. Let $V\to U$ be an étale Galois cover such that:\begin{itemize}
\item each sheaf $\LL^s|_V$, $s\in \{0,\dots, t\}$ is constant;
\item each map $\HH^1(U,\LL^s)\to \HH^1(V,\LL^s|_V)$, $s\in\{0,\dots, t\}$ is trivial;
\item the ramification index of $V\to U$ above every point of $\tX-U$ is divisible by $n$.
\end{itemize} Denote by $G$ the group $\Aut(V|U)$. Let $\tilde Z$ be the preimage of $Z$ in $\tX$. Given $\tilde z\in\tilde Z$, denote by $I_{\tilde z}\subset G$ its inertia group, and by $P_{\tilde z}\triangleleft I_{\tilde z}$ its wild inertia subgroup. For each integer $s\in \{0,\dots,t\}$, let $\phi^s\colon \F^s\to j_\star \LL^s$ be the adjunction unit. Denote by $\partial_M$ the transition maps of the complex $M^\bullet$, by $\partial_\F$ those of $\F^\bullet$, and by $\partial_G$ the coboundary maps in the cochain complex representing group cohomology with respect to the group $G$; in particular, given an element $m$ in a $G$-module, $\partial_G(m)$ is the crossed homomorphism $g\mapsto g\cdot m-m$. The remainder of this section will be dedicated to the proof of the following result.

\begin{theorem}\label{th:RGamma} In this situation, $\RG(X,\F^\bullet)[1]\in \DD^b_c(\Lambda)$ is represented by the cone of the following morphism of complexes, whose terms are indexed by $s\geqslant 0$:
\[
\begin{tikzcd}
\cdots\arrow[r] & M^s\oplus\Homcr(G,M^{s-1})\oplus\bigoplus_{z\in Z} \F_z^{s}\oplus \bigoplus_{\tilde z\in \tZ}\HH^1(I_{\tilde z}/P_{\tilde z},M_{P_{\tilde z}}^{s-2}) \arrow[d] \arrow[r]&\cdots \\
\cdots\arrow[r] & \bigoplus_{\tilde z\in \tZ}\left(M^s\oplus \Homcr(I_{\tilde z}/P_{\tilde z},M^{s-1}_{P_{\tilde z}})\oplus\HH^1(I_{\tilde z}/P_{\tilde z},M^{s-2}_{P_{\tilde z}})\right) \arrow[r]&\cdots
\end{tikzcd}
\]
Here, $(a,b,c_z,d_{\tilde z})$ is sent to $(a-\phi_{z}(c),\res_G^{I_z}(b),d)$ where $z=\nu(\tilde z)$ and $\res_G^{I_z}$ denotes the composite map $\Homcr(G,M^s)\to\Homcr(I_z,M)^s\to \Homcr(I_z/P_z,M^s_{P_z})$ defined in \Cref{lem:secIP}.
Then the transition map of the top complex is $(a,b,c_z,d_{\tilde z})\mapsto (\partial_M(a),\partial_M(b)+(-1)^s\partial_G(a),\partial_\F(c_z),\partial_M(d_{\tilde z})+(-1)^s\res_G^{I_z}(b))$. The transition map of the bottom complex is $(a,b,c)\mapsto (\partial_M(a),\partial_M(b)+(-1)^s\partial_G(a),\partial_M(c)+(-1)^s\partial_G(b))$.
\end{theorem}

\begin{cor} When computing the cohomology of a single constructible sheaf $\F$ on $X$ with generic geometric fibre $M$, the complex $\RG(X,\F)[1]$ is the cone of the following morphism of complexes:
\[
\begin{tikzcd}
(\oplus_{z\in Z}\F_z) \oplus M \arrow[r,"{(0,\partial_G)}"]\arrow[d,"{((\phi_z-\id)_{\tilde z\to z})_z}"] & \Homcr(G,M)\arrow[r,"{(\res_G^{I_{\tilde z}})_{\tilde z}}"]\arrow[d,"{(\res_G^{I_{\tilde z}})_{\tilde z}}"] & \oplus_{\tilde z\in\tilde Z}\HH^1(I_{\tilde z}/P_{\tilde z},M_{P_{\tilde z}})\arrow[d,"{\id}"] \\
\oplus_{\tilde z\in\tilde Z} M \arrow[r,"{\partial_{I_z}}"] &\oplus_{\tilde z\in\tilde Z}\Homcr(I_{\tilde z}/P_{\tilde z},M_{P_{\tilde z}}) \arrow[r] & \oplus_{\tilde z\in\tilde Z}\HH^1(I_{\tilde z}/P_{\tilde z},M_{P_{\tilde z}})
\end{tikzcd}
\]
\end{cor}

\begin{proof}[Proof of the theorem]The functors $j_\star,j^\star,i_\star,i^\star$ can be extended to the (non-derived) category of complexes of constructible sheaves on $X$. These functors will allow us to compute  $\RG(X,\F^\bullet)$ in the following way: first consider a short exact sequence involving $\F^\bullet$ in this category of complexes of constructible sheaves on $X$, then compute the associated distinguished triangle in $\DD^b_c(\Lambda)$. \\

Consider the following short exact sequence of complexes of constructible sheaves on $X$:
\[ 0\to \F^\bullet \to j_\star \mathscr{L}^\bullet\oplus i_\star i^\star \F^\bullet \to i_\star \mathcal{Q}^\bullet \to 0\]
where $\mathcal{Q}^\bullet\coloneqq i^\star j_\star \mathscr{L}^\bullet$. Here, the first map is just the sum of the two adjunction maps $\F^\bullet \to j_\star j^\star \F^\bullet$ and $\F^\bullet \to i_\star i^\star\F^\bullet$, and the second map is the difference of the adjunction maps $j_\star \LL\to i_\star i^\star j_\star \LL$ and $i_\star i^\star \F^\bullet \to i_\star i^\star j_\star j^\star \F^\bullet$.
The object $\RG(X,\F^\bullet)[1]$ of $\DD^b_c(\Lambda)$ is the cone of the morphism \[\RG(X,j_\star \mathscr{L}^\bullet)\oplus\RG(X,i_\star i^\star \F^\bullet)\to \RG(X,i_\star \mathcal{Q}^\bullet).\]

\paragraph{Computing $\RG(X,j_\star \mathscr{L})$} The following distinguished triangle in $\DD^b_c(X,\Lambda)$ \cite[09XP]{stacks}:
\[ \RG_Z(X,j_\star \mathscr{L}^\bullet)\to \RG(X,j_\star \mathscr{L}^\bullet)\to \RG(U,\mathscr{L}^\bullet)\xrightarrow{+1} \]
shows that $\RG(X,j_\star \mathscr{L}^\bullet)[1]$ is the cone of $\RG(U,\mathscr{L}^\bullet)\to \RG_Z(X,j_\star \mathscr{L}^\bullet)[1]$. Let us now turn to the computation of $\RG(U,\mathscr{L}^\bullet)$ and $\RG_Z(X,j_\star \mathscr{L}^\bullet)$. 
According to \Cref{prop:H1cd1}, $\RG(U,\mathscr{L}^\bullet)$ is represented by the total complex associated to the double complex 
\[\begin{tikzcd}
\Homcr(G,M^0)\arrow[r] & \Homcr(G,M^1) \arrow[r]& \Homcr(G,M^2) \arrow[r] & \cdots \\
M^0 \arrow[u] \arrow[r] & M^1\arrow[u]\arrow[r] & M^2 \arrow[u]\arrow[r] & \cdots 
\end{tikzcd} \]
which is 
\[  M^0\to M^1\oplus\Homcr(G,M^0)\to M^2\oplus\Homcr(G,M^1)\to \dots \]
For each integer $s\in \{0,\dots, t\}$, $\RG_Z(X,j_\star \mathscr{L}^s)$ is represented by the one-term complex \[\bigoplus_{\tilde z\in \tilde Z}\HH^2_{\tilde z}(X,\nu^\star j_\star\LL^{s})[-2] \] according to \Cref{lem:cohsuppz}. By \Cref{cor:RGZ} and \Cref{cor:cohI}, $\HH^2_{\tilde z}(X,\nu^\star j_\star\LL^{s})=\HH^1(I_{\tilde z},M^s)=\HH^1(I_{\tilde z}/P_{\tilde z},M^s_{P_{\tilde z}})$. It follows that $\RG(X,j_\star\LL^\bullet)$ is represented by the complex 
\[\adjustbox{max width=\textwidth}{$ M^0\xrightarrow{(\partial_M^0,\partial_G^0)} M^1\oplus\Homcr(G,M^0)\xrightarrow{(\partial_M^1,\partial_G^0-\partial_M^0,\res_G^{I_{\tilde z}})} M^2\oplus\Homcr(G,M^1)\oplus\bigoplus_{\tilde z\in \tilde Z}\HH^1(I_{\tilde z},M^0)\to \dots$} \]

\paragraph{Representing $\RG(X,i_\star i^\star\F^\bullet)$} For each integer $s\in \{0,\dots, t\}$, $\RG(X,i_\star i^\star\F^s)$ is simply represented by the one-term complex $\HH^0(Z,\F^s)[0]$. Therefore $\RG(X,i_\star i^\star\F^\bullet)$ is represented by the complex \[ \bigoplus_{z\in Z}\F^0_z\to\bigoplus_{z\in Z} \F^1_z\to \bigoplus_{z\in Z}\F^2_z\to\cdots\]
whose differentials are those of $\F^\bullet$.

\paragraph{Representing $\RG(X,i_\star \mathcal{Q}^\bullet)$} We know that $\RG(X,i_\star\mathcal{Q}^s)=\RG(X,i_\star i^\star j_\star\LL^s)$ is represented by the one-term complex $\bigoplus_{z\in \tZ}\HH^0(I_z,M^s)$ concentrated in degree zero. However, in order to be able to express the map $\RG_Z(X,j_\star\LL)\to\RG(X,i_\star\mathcal{Q}^s)$ as a morphism of complexes, we will write $\RG(X,i_\star \mathcal{Q}^s)$ as the following three-term complex:
\[ \bigoplus_{{\tilde z}\in \tZ}M^s\to \bigoplus_{{\tilde z}\in\tZ}\Homcr(I_{\tilde z}/P_{\tilde z},M^s_{P_{\tilde z}})\to \bigoplus_{{\tilde z}\in \tZ}\HH^1(I_{\tilde z}/P_{\tilde z},M^s_{P_{\tilde z}})\]
whose first differential is the one sending $(m^s_{\tilde z})_{\tilde z}$ to $(g_{\tilde z}\mapsto (g_{\tilde z}\cdot m_{\tilde z}^s-m_{\tilde z}^s))_{\tilde z}$, and the second one is the usual quotient map.
By the same arguments as above, $\RG(X,i_\star \mathcal{Q}^\bullet)$ is thus represented by the complex
\[\adjustbox{max width=\textwidth}{$ \bigoplus_{{\tilde z}\in \tZ}M^0\to \bigoplus_{{\tilde z}\in\tZ}(M^1\oplus\Homcr(I_{\tilde z}/P_{\tilde z},M^0_{P_{\tilde z}}))\to \bigoplus_{{\tilde z}\in \tZ}(M^2\oplus \Homcr(I_{\tilde z}/P_{\tilde z},M^1_{P_{\tilde z}})\oplus\HH^1(I_{\tilde z}/P_{\tilde z},M^0_{P_{\tilde z}}))\to\cdots$}\]

\paragraph{Putting everything together} We have now computed each term of the morphism of complexes
\[\RG(X,j_\star \mathscr{L}^\bullet)\oplus\RG(X,i_\star i^\star \F^\bullet)\to \RG(X,i_\star \mathcal{Q}^\bullet)\]
whose cone is $\RG(X,\F^\bullet)$. The morphism itself is the difference of the two adjunction maps.
\end{proof}

\begin{rk}\label{rk:fieldext}
A cover $V\to U$ as in the theorem may be computed in the following way: consider an étale Galois cover $W\to U$ such that each $\LL^i|_W$ is constant, and set $V=W\nt$.
While $W\nt$ is the canonical choice, we may for complexity reasons choose any of its subcovers that still satisfy the three required properties listed above. Here is one example of such a subcover. The action of $\Gk$ on $\HH^1(U,\LL^\bullet)$ factors through a finite quotient $\Gal(k_1|k_0)$. The image of $\HH^1(U,\LL^\bullet)$ in $\HH^1(W,\Lambda)$ is still defined over $k_1$, and we may construct a Galois subcover of $W\nt\to U$ by taking $n^{\text{th}}$ roots only of functions that are defined over $k_1$. This will be used in the case of finite fields in section \ref{subsec:ffieldcoh}.
\end{rk}

\subsection{Computing the Galois action}\label{subsec:galnongeomcon}

Let $k_0$ be a perfect field, let $k$ be an algebraic closure of $k_0$ and $\GO=\Gal(k|k_0)$. Consider a curve $X_0$ over $k_0$. For the sake of simplicity, we consider a single sheaf $\F_0$ on $X_0$. The general case of complexes of sheaves is handled in the same way. Let $X=X_0\times_{k_0}k$ and $\F=\F_0|_X$. This section is dedicated to the description of a complex of $\Lambda[\GO]$-modules representing $\RG(X,\F)$. Let $U_0$ be an affine open subset of $X_0$ on which $\F$ is locally constant, and $W_0\to U_0$ be a Galois cover such that $\F_0|_{W_0}$ is constant. Let $Z_0$ be the reduced closed complement of $U_0$ in $X_0$. Denote by $U,W,Z$ their base changes to $k$. Consider the cover $V=W\nt$, and the complex representing $\RG(X,\F)$ computed above using $V$.

\paragraph{The easy case: when $W_0$ has a rational point $w_0$} In that case, $W$ is connected. The terms of $\RG(X,\F)$ consist of cohomology groups of $G=\Aut(V|U)$ or $I_z/P_z$, where $z$ is a closed point at infinity of $U$, with values in (a suitable quotient of) $M$. Understanding the action of $\GO$ on $G$ is straightforward: let $\bar w$ be a geometric point of $W$ whose image in $W$ is $w_0$, let $\bar u$ be its image in $U$ and  $\bar v$ a preimage of $\bar w$ in $V=W\nt$. The group $\GO$ acts on $\pi_1(W,\bar w)\subset \pi_1(U,\bar u)$ by functoriality, and since $W\nt\to W$ is characteristic, this action restricts to $\pi_1(V,\bar v)$. 
Moreover, for $z\in X(k_0)$, the group $\GO$ acts naturally on $I_z/P_z$, which is canonically isomorphic to $\mu_e(k)$, where $e$ is the ramification index of $V\nt\to U$ above $z$. Finally, the group $M^{P_z}$ does not depend on the choice of preimage of $Z$ in the compactification of $V\nt$.
For a point $z$ defined over an extension $k_1$ of $k_0$, we need to consider its Galois orbit $T$: then $\GO$ acts naturally on $\oplus_{t\in T}\RG(I_t/P_t,M^{P_t})$, where the preimages of $t$ whose inertia groups we compute have been chosen in the same $\GO$-orbit. These considerations allow us to compute the action of $\GO$ on each of the terms of the complex representing $\RG(X,\F)$.

\paragraph{The general case} In general, $W$ need not be connected, and the function field $k_0(W_0)$ may contain a finite nontrivial Galois extension of $k_0$. Section \ref{subsec:X2gal} shows how to construct a Galois cover $V_0\to U_0$ whose function field contains a sufficiently large Galois extension $L$ of $k_0$ over which the elements of $\HH^1(W,\mu_n)$ as well as the points at infinity of $V\coloneqq V_0\times_{k_0}k$ are defined. This ensures that any closed point of the smooth compactification of $V_0$ above a point of $Z_0$ is exactly $L$. 
Consider a connected component $V_c$ of $V$. The group $G_c\coloneqq\Aut(V_c|U)=\ker(\Aut(V_0|U_0)\to \Gal(L|k_0))$ is the stabiliser of $V_c$ in $G\coloneqq\Aut(V_0|U_0)$. 
Set $M_0= \HH^0(V_0,\F_0)$ and $M= \HH^0(V,\F)$; then $M$ is the induced representation $\ind_{G_c}^G (M_0)$, and Shapiro's lemma \cite[Th. 4.19]{neukirch_cft} shows that the map \[ \tau_{\leqslant 1}\RG(G_c,M_0)\to \tau_{\leqslant 1}\RG(G,M)\]
in $\DD^b_c(\Lambda)$ is a quasi-isomorphism. As an abelian group, $M=M_0^d$ where $d$ is the number of connected components of $V$. The group $\GO$ acts naturally on $M$ as it does on set the connected components of $V$, in a way that is compatible with its action on $G$. This allows to compute the action of $\GO$ on $\tau_{\leqslant 1}\RG(G,M)$. 
Let $z\in Z_0$ be a closed point with residue field $k_1$. 
Let $v$ be a preimage of $z$ in $V_0$. Denote by $T\coloneqq \{\tau(z),\tau\in \Gal(k_1|k_0)\}$ the $\GO$-orbit of $z$. For each $\tau\in \Gal(k_1|k_0)$, consider a preimage $v_{\tau}$ of $\tau(z)$ in $V_0$; we choose the $v_{\tau}$ to be in the same $\GO$-orbit.
Let $D_\tau\triangleleft G$ be the decomposition group of $v_{\tau}$. The map $D_\tau\to \Gal(L'|k_0)$ is surjective, and its kernel is the inertia group $I_\tau$ of $v_{\tau}$, so that $M=\ind_{I_\tau}^{D_\tau}M_0$. Shapiro's lemma now ensures that \[ \RG(D_\tau,M)=\RG(I_\tau,M_0)=\RG(I_\tau/P_\tau,M_0^{P_\tau}) \] in $\DD^b_c(\Lambda)$. The group $\GO$ acts naturally on the complex $\bigoplus_{\tau}\RG(D_\tau,M)$, whose cohomology groups are $\HH^0(T,j_\star\F|_U)$ and $\HH^2_T(X,j_\star\F|_U)$. Using these notations, $\RG(X,\F)[1]$ is isomorphic in $\DD^b_c(\Lambda[\GO])$ to the cone of the following morphism of complexes:
\[
\begin{adjustbox}{width=\textwidth}{
\begin{tikzcd}
 0\arrow[r] & M\oplus \bigoplus_{T} \HH^0(T,\F) \arrow[d,"\bigoplus_{T,\tau}(\res_{I_\tau}^G-\phi_z)"] \arrow[r,"{(\partial_G,0)}"] & \Homcr(G,M) \arrow[d,"\bigoplus_{T,\tau} \res_{I_\tau}^G"]\arrow[r] & \bigoplus_{T}\bigoplus_{\tau}\HH^1(I_\tau/P_\tau,M_{P_\tau}) \arrow[d,"\id"]\arrow[r] & 0 \\
 0 \arrow[r]& \bigoplus_{T}\bigoplus_{\tau}M_{P_\tau} \arrow[r,"\bigoplus_z\partial_{I_\tau}"] \arrow[r] & \bigoplus_{T}\bigoplus_{\tau}\Homcr(I_\tau/P_\tau,M_{P_\tau}) \arrow[r] & \bigoplus_{T}\bigoplus_{\tau}\HH^1(I_\tau/P_\tau,M_{P_\tau}) \arrow[r] & 0
\end{tikzcd}
}
\end{adjustbox}
\]
where $T$ runs through the $\GO$-orbits in $Z$, and $\tau$ runs through the $k_0$-automorphisms of the residue field of the closed points of $T$.

\subsection{Functoriality over $\Spec k$}

Consider a morphism $\phi\colon X'\to X$ of integral curves over an algebraically closed field $k$. As usual, let $n$ be an integer invertible in $k$, and denote by $\Lambda$ the ring $\ZZ/n\ZZ$. Let $\F$ be a constructible sheaf of $\Lambda$-modules on $X$. Here is how to compute a morphism of complexes of $\Lambda$-modules representing $\RG(X,\F)\to \RG(X',\phi^\star\F)$. Let $U$ be an affine open subset of $X$ on which $\F$ is locally constant. Let $W\to U$ be an étale Galois cover such that $\F|_W$ is constant, and $W'\to U'$ be the Galois closure of a connected component of $W\times_X X'$. Consider the Galois covers $V=W\nt\to W$ and $V'=(W')\nt\to W'$. Given the construction of $V$ and $V'$, the map $\HH^1(W,\mu_n)\to\HH^1(W',\mu_n)$ defines a map $V'\to V$.

By elementary Galois theory, there is a map $\Aut(V'|U')\subseteq \Aut(V'|U)\to\Aut(V|U)$. For each point $z\in X-U$, choose a preimage $z_V$ of $z$ in the smooth compactification of $V$. For each preimage $z'$ of $z$ in $X'$, consider a preimage $z_V'$ of $z'$ in the smooth compactification of $V'$ whose image in $V$ is $z_V$. Consider the inertia groups $P_V\triangleleft I_V\subseteq\Aut(V|U)$ of $z_V$ and $P_{V'}\triangleleft I_{V'}\subseteq\Aut(V'|U')$ of $z_V'$. The map $\Aut(V'|U')\to\Aut(V|U)$ induces for each choice of $z,z',z_V,z_V'$ a map $I_{V'}/P_{V'}\to I_V/P_V$. The functoriality of the bar resolution thus allows to compute the maps $\RG(\Aut(V'|U'),M)\to\RG(\Aut(V|U),M)$ and $\RG(I_V/P_V,M^{P_V})\to\RG(I_{V'}/P_{V'},M^{P_{V'}})$ that are needed to compute the morphism of complexes representing $\RG(X,\F)\to\RG(X',\phi^\star\F)$.

\subsection{Cohomology of curves over a field of cohomological dimension 1}

Consider a perfect field $k_0$ of cohomological dimension 1 such that $\HH^1(k_0,\ZZ/n\ZZ)$ is finite, e.g. a finite field or the fraction field of a strictly henselian local DVR. Let $n$ be an integer invertible in $k_0$. Denote by $\Lambda$ the ring $\ZZ/n\ZZ$. Let $X_0$ be an integral curve over $k_0$, and $\F_0^\bullet$ be a complex of constructible sheaves of $\Lambda$-modules on $X_0$. Denote by $k$ the algebraic closure of $k_0$, by $\GO$ the Galois group $\Gal(k|k_0)$ and by $X,\F$ the respective base changes of $X_0,\F_0$ to $k$. \Cref{th:RGamma} shows how to compute a complex $M^\bullet$ of $\GO$-modules representing $\RG(X,\F^\bullet)$. Recall that \[ \RG(X_0,\F_0^\bullet)=\RG(\GO,\RG(X,\F^\bullet)).\]
We described in section \ref{subsec:cd1} how to determine a complex representing this object. Let $k_1$ be a Galois extension of $k$ such that the action of $\GO$ on $\RG(X,\F^\bullet)$ factors through $\Gal(k_1|k_0)$. Consider the extension $k_1\nt$ of $k_0$ with Galois group $\HH^1(k_1,\Lambda)^\vee$, and the Galois group $G=\Gal(k_1\nt|k_0)$.
Then $\RG(X_0,\F_0)$ is represented by the total complex associated to the double complex $B^{i,j}=\Hom_{\Lambda[G]}(\tau_{\geqslant -1}P_{G}^{-j}(\Lambda),M^i)$, where $P_G$ is the usual projective resolution of $\Lambda$ as a $\Lambda[G]$-module.

\begin{rk} The same method also applies in theory to the general case where $\HH^1(k_0,\Lambda)$ is infinite, using continuous group cohomology and $\Lambda[[G]]$-modules; one particularly interesting case to consider would be when $k_0$ is the function field of a curve over an algebraically closed field. However, to go any further in practical computations, one is quickly confronted with the issue of computing a generating set of $\HH^1(k_0,\ZZ/n\ZZ)\simeq k_0^\times/(k_0^\times)^n$.
\end{rk}

\section{Algorithmic aspects}\label{sec:algorithmic}

In this whole section, $n$ is a positive integer and $\Lambda$ denotes the ring $\ZZ/n\ZZ$. 

\subsection{Representing curves and sheaves}\label{subsec:rep}

\paragraph{Representing curves} A smooth projective curve over a field $k_0$ is defined by a (possibly singular) plane model given by a polynomial in two variables. When working with a closed subscheme of a smooth curve, we may always suppose that its image in the plane model is nonsingular. Such a closed subscheme is defined by equations; an open subscheme is defined by its closed complement. A morphism of smooth curves is given by a morphism of plane models, i.e. by two polynomials in two variables. The only time we need to work with rational points is when considering the geometric points in a given closed subscheme $Z_0$ of the curve; in that case, we may replace $k_0$ with a finite extension over which these points are defined. This extension has degree bounded by the number $r$ of geometric points in $Z_0$, and passing to this extension has no impact on the complexity estimates given below, which are all at least polynomial in $r$. 
A curve $X$ with multicross singularities is defined by its normalisation $\tX$, as well as the subsets of points of $\tX$ that have the same singular image in $X$; again, we suppose that these subsets of $\tX$ have a nonsingular image in the plane model. 

\paragraph{Representing sheaves} Let $X$ be a smooth curve over an algebraically closed field $k$ of characteristic prime to $n$. A constructible sheaf of $\Lambda$-modules on the étale site of $X$ will be described by the following gluing data, which defines it uniquely \cite[II, Th. 3.10]{milneEC}: \begin{itemize}[label=$\bullet$]
\item a closed 0-dimensional subscheme $Z$ of $X$, defined by an equation;
\item a finite locally constant sheaf $\LL$ on the open complement $U$ of $Z$, defined by a Galois cover $V\to U$ and the action of the group $G=\Aut(V|U)$ on the $\Lambda$-module $M=\HH^0(V,\LL)$;
\item for each point $z\in Z$, the $\Lambda$-module $\F_z$ defined by generators and relations;
\item for each point $z\in Z$, the gluing morphism $\phi_z\colon \F_z\to (j_\star\LL)_z=\HH^0(I_z,M)$, where $I_z\subset G$ is the stabiliser of a preimage of $z$ in $V$.
\end{itemize}
This data also allows to represent morphisms and direct sums of sheaves in a very straightforward manner, as well as to compute tensors products and $\Hom$-sheaves; see \cite[§III.4]{mathese} for more details. While this representation of constructible sheaves might not be the first that comes to mind, it is well suited to our computation of cohomology groups. The usual ways of defining constructible sheaves (as cokernel of $f_!\Lambda\to g_!\Lambda$ with $f,g$ étale or kernel of $f_\star \Lambda\to g_\star\Lambda$ with $f,g$ finite) also admit an algorithmic representation, which can be converted into this one (see \cite[§III.3]{mathese}).

\subsection{Computing the cohomology of $\mu_n$: existing algorithms}\label{subsec:exalg}

Our methods rely on existing algorithms which, given a smooth integral curve $X$ over an algebraically closed field $k$ of characteristic prime to $n$, compute $\HH^1(X,\mu_n)$. Recall that we need to be able to compute it even for affine curves, which can prove to be a bit trickier than in the projective case.

The most efficient algorithm computing $\HH^1(X,\mu_n)$, developed by Couveignes, only applies to projective curves over finite fields, and actually requires prior knowledge of the characteristic polynomial of the Frobenius endomorphism of $X$; since it makes use of some properties of the Frobenius and the group structure of $\Pic^0(X)[n]$, adapting it to the cohomology of affine curves, or of curves over other types of fields does not seem easy. Given a curve of genus $g$ over $\FF_q$, described by an ordinary plane model of degree $d$, it computes $\Pic^0(X)(\FF_q)[n]$ in time polynomial in $d, g,\log q, n$ \cite[Th. 1]{couveignes_linearizing}. 

While it was also first described only for projective curves over finite fields, Huang and Ierardi's method \cite{huang_counting} applies to more general settings. Their algorithm constructs an affine scheme whose points correspond to divisors $D$ such that $nD$ is the divisor of a rational function, and then finds a point in each irreducible component of this scheme, which is enough to find a representative of every $n$-torsion class in $\Pic^0(X)$. This strategy readily adapts to the computation of division by $n$ in $\Pic^0(X)$, thus allowing to compute the cohomology of $\mu_n$ on an open subset of $X$. It is also independent of the chosen base field. The complexity of their algorithm, computed when the base field is $\FF_q$, is polynomial in $n^g,n^d,\log q$ \cite[Prop. 4.3.3]{mathese}.

In the remainder of this article, we will denote by $\mathsf{H1Const}(k_0,n,d,g,r)$ or simply $\mathsf{H1Const}(X,n)$ the complexity of the computation of $\HH^1(X,\mu_n)$, where $X$ is a smooth integral curve of genus $g$ over $k$, given by a degree $d$ polynomial in $k_0[x,y]$, with $r$ points at infinity. We will also denote by $\mathsf{Root(k_0,n,d)}$ the complexity of computing an $n^{\text{th}}$ root of an element in a degree $d$ extension of $k_0$. 

\subsection{Computing in $\HH^1(X,\mu_n)$}

Let $U$ be a smooth integral curve over an algebraically closed field $k$ of characteristic prime to $n$. Let $X$ be the smooth compactification of $U$. Let $n$ be an integer invertible in $k$, and $\Lambda$ be the ring $\ZZ/n\ZZ$.  Denote by $P_0,\dots,P_{r}$ the points of $X-U$. Recall that the elements of $\HH^1(U,\mu_n)$ are equivalence classes $[D,f]$ of pairs where $f$ is a rational function on $U$ such that $\div_U(f)=nD$. The class $[D,f]$ is trivial in $\HH^1(U,\mu_n)$ precisely when $f$ is an $n^{\text{th}}$ power. The sum $[D,f]+[D',f']$ is defined by $[D+D',ff']$.

Here is how to compute the coordinates of an element of $\HH^1(U,\mu_n)$ in a given basis using the Weil pairing, as is done in the projective case in \cite[§8]{couveignes_linearizing}. The Weil pairing \[ e_n\colon \HH^1(U,\mu_n)\times \HH^1_c(U,\mu_n)\to \mu_n\]
is nondegenerate. In this context, \[ \HH^1_c(U,\mu_n)= \frac{\{(D,f)\in \Div(U)\times k(X)^\times\mid nD=\div(f), f(P_0)=...=f(P_r)=1\}}{\{(D,f^n)\text{ where }f(P_0)=...=f(P_r)=1\}}\]
sits in the short exact sequence 
\[ 0\to \frac{\mu_n(k)^r}{\mu_n(k)}\to \HH^1_c(U,\mu_n)\to \HH^1(X,\mu_n)\to 0\]
and may be computed in the following way. Choose a primitive $n^{\text{th}}$ root of unity $\zeta\in k$. For each $i\in \{1\dots r\}$, consider a function $g_i\in k(X)$ such that $g_i(P_0)=\zeta^{-1}$, $g_i(P_i)=\zeta$, and $g_i(P_j)=1$ for all $j\neq 0,i$. Then $([\div(g_1),g_1^n],\dots,[\div(g_{r},g_r^n])$ is a basis of the image of $\mu_n(k)^r/\mu_n(k)$ in $\HH^1_c(U,\mu_n)$. The Weil pairing is computed as usual: given $v=[D,f]\in \HH^1(U,\mu_n)$ and $w=[E,g]\in \HH^1_c(U,\mu_n)$ where $f$ and $g$ are suitably normalised, \[ e_n(v,w)=\frac{f(E)}{g(D)}.\]

\begin{algorithm}[H]\label{alg:coord}
\SetAlgoLined  
\caption{\textsc{CoordinatesInBasis}}
\KwData{Smooth integral curve $U$ over alg. closed field $k$, smooth compactification $X$ of $U$\\
The points $P_0,\dots,P_{r}$ of $X-U$\\
Positive integer $n$ invertible in $k$\\
Basis $B=(v_1,\dots,v_{2g+r})$ of $\HH^1(U,\mu_n)$, where $v_i=(D_i,f_i)\in \Div(U)\times K^\times$ and $([v_1],\dots,[v_{2g}])$ is a basis of $\HH^1(X,\mu_n)$  \\
Element $v_0\in \HH^1(U,\mu_n)$ represented by $(D_0,f_0)\in \Div(U)\times K^\times$\\
Primitive $n^{\text{th}}$ root of unity $\zeta\in k$}
\KwResult{Coordinates $(\alpha_1,\dots,\alpha_{2g+r})\in\Lambda^{2g+r}$ of $v$ w.r.t. $B$\\
Function $h\in k(X)^\times$ such that $f=h^nf_1^{\alpha_1}\dots f_{2g+r}^{\alpha_{2g+r}}$}
\hrulefill \\
\For{$i\in \{1\dots r\}$}{
	Compute function $f_i$ such that $f_i(P_0)=\zeta^{-1}$, $f_i(P_i)=\zeta$, and $f_i(P_j)=1$ for $j\neq 0,i$\\
	Set $v_{1-i}:=(\div(g_i),g_i^n)$
	}
	Compute matrix $M=e_n(v_i,v_j)_{0\leqslant i\leqslant 2g+r, 1-r\leqslant i\leqslant 2g}\in \Mat_{(2g+r+1)\times (2g+r)}(\mu_n(k))$\\
	Compute an element $(1,-\alpha_1,\dots,-\alpha_{2g})\in\ker(M)$: then $v_0=\sum_i{\alpha_i}v_i$ in $\HH^1(U,\mu_n)$\\
	Compute Riemann-Roch space $L$ of $D_0-\sum_i\alpha_iD_i$\\
	Pick $h\in L$: then $D_0-\sum_i\alpha_iD_i=\div(h^{-1})$\\
	Compute $n^{\text{th}}$ root $c\in k$ of $f_0h^n\prod_if_i^{-\alpha_i}\in k$\\
	\Return $\alpha_1,\dots,\alpha_{2g+r}\in \Lambda$ and function $ch\in k(X)^\times$
\end{algorithm}

\begin{lem}
Suppose all of the divisors $D_0,\dots,D_{2g+r}$ are given as difference of two effective divisors of degree $\leqslant m$, and the curve $X$ is given by a plane model of degree $d$. This algorithm returns the coordinates of $v$ in time $\Poly(d,m,g,r)+\mathsf{Root}\left(k_0,n,n^{(2g+r)^2}\right)$.
\begin{proof} Computing the functions $f_i$ using Lagrange interpolation, as well as the matrix $M$ using the definition above, is straightforward. The kernel of $M$ is computed using standard linear algebra techniques over $\Lambda$ (using the isomorphism $\mu_n(k)\to\Lambda$ given by $\zeta\mapsto 1$), which run in polynomial time in the size of $M$. 
\end{proof}
\end{lem}

\subsection{Construction of $\Vnt$}

Let $k$ be an algebraically closed field of characteristic prime to $n$. Let $X$ be an integral smooth projective curve over $k$, and $U$ an affine open subscheme of $X$. Let $V$ be an étale Galois cover of $U$. The following algorithm computes the cover $\Vnt\to V$ defined in section \ref{subsec:1.1}, as well as the group $\Aut(\Vnt|U)$.

\begin{algorithm}[H]\label{alg:X2}
\SetAlgoLined  
\caption{\textsc{nTorsCover}}
\KwData{Galois cover $V\to U$ of smooth integral curves over alg. closed field $k$ \\
Generating set $S$ of $\Aut(V|U)$\\
Integer $n$ invertible in $k$}
\KwResult{Generating set of $\Aut(\Vnt|U)$}
\hrulefill \\
Compute basis $[D_i,f_i]_{1\leqslant i\leqslant s}$ of $\HH^1(V,\mu_n)$ (see section \ref{subsec:exalg})\\
\For{$\sigma\in S$}{
	\For{$i \in \{1\dots s\}$}{
	Compute $\sigma^\star (D_i,f_i)$\\
	Compute $h, \alpha_1,\dots,\alpha_s$ such that $\sigma^\star f_i=h_i^nf_1^{\alpha_1}\dots f_s^{\alpha_s}$ using Algorithm \ref{alg:coord}
	}
	Define $\rho_\sigma\colon (x,y)\mapsto \sigma(x,y), z_j\mapsto h_jz_1^{\alpha_1}\dots z_s^{\alpha_s}$
}
\For{$i \in \{1\dots s\}$}{
	Define $\phi_i\colon (x,y)\mapsto (x,y), z_i\mapsto \zeta z_i, (z_j)_{j\neq i} \mapsto (z_j)_{j\neq i}$
}
\Return $\{ \rho_\sigma\}_{\sigma\in S}\cup\{ \phi_i\}_{1\leqslant i\leqslant s}$
\end{algorithm}

\begin{prop} If $U$ has $r$ points at infinity and $V$ is given by an ordinary plane model of degree $d$, Algorithm \ref{alg:X2} computes a generating set of $\Aut(\Vnt|U)$ in \[ \mathsf{H1Const}(k_0,n,d,(2g+r)[V:U],r[V:U])\]
elementary operations. 
\begin{proof}
The genus of $V$ is bounded by $(2g+r)[V:U]$. The complexity of computing the coordinates of the pullback of the divisors is polynomial-time in $n,[V:U](2g+r),d$, hence dominated by that of computing a basis of $\HH^1(V,\mu_n)$.
\end{proof}
\end{prop}

Here is how, once $G\nt=\Aut(V\nt\to U)$ has been computed, to compute the preimages of points of $X$ in the smooth compactification of $\Vnt$, as well as their inertia group. This is done by considering a suitable explicit model of $\Vnt$. Recall that the function field of $\Vnt$ is obtained from that of $V$ by adjoining $n^{\text{th}}$ roots of functions $f_1,\dots,f_t\in k(V)$. Write $f_i=\frac{g_i}{h_i}$, where $g_i,h_i\in k[x,y]$. Denote by $Z$ (resp. $W$, resp. $W\nt$) the sets of points at infinity of $U$ (resp. $V$, resp. $V\nt$). Consider a point $z\in Z$, and a preimage $w$ of $z$ in $W$. Replacing $f_i$ with $h_i^nf_i$ if necessary, we may suppose $h_i(w)\neq 0$. Then the affine curve given by the equation of $V$ and $h_iz_i^n-g_i$ contains $n^{t-1}$ preimages of $w$, which are nonsingular. Given one of these preimages $w$, the inertia subgroup $I_{w}$ can be computed simply by evaluating the elements of $G\nt$ at $w$.

\begin{algorithm}[H]\label{alg:inertia}
\SetAlgoLined  
\caption{\textsc{InertiaGroup}}
\KwData{Galois cover $V\to U$ of smooth integral curves over alg. closed field $k$, with smooth compactification $Y\to X$
Integer $n$ invertible in $k$\\
Group $\Aut(\Vnt|U)$ and basis $(D_i,f_i=\frac{g_i}{h_i})_{1\leqslant i\leqslant s}$ of $\HH^1(V,\mu_n)$\\
Point $z$ in compactification of $U$, preimage $w$ of $z$ in compactification of $V$\\
}
\KwResult{Preimage $w\nt$ of $w$ in compactification of $\Vnt$\\ 
Generating set of inertia group $I\nt\subset \Aut(\Vnt|U)$ of $w\nt$}
\hrulefill \\
\For{$i\in \{1\dots s\}$}{
	\If{$h_i(w)=0$}{$g_i\leftarrow h_i^{n-1}g_i, h_i\leftarrow 1$}
	Compute root $t_i$ of $h_i(z)T^n-g_i(z)\in k[T]$\\
}	
Set $w\nt=(w,t_1,\dots,t_s)\in \Spec Y[z_1,\dots,z_n]/(h_iz_i^n-g_i)$\\
$I\nt\coloneqq\{ \sigma\in G\nt\mid \sigma(w\nt)=w\nt\}$\\
\Return $w\nt,I\nt$
\end{algorithm}

\begin{lem} Algorithm \ref{alg:inertia} returns a preimage $w\nt$ of $w$ in the smooth compactification of $\Vnt$ and its stabiliser in \[[V:U]\left((n^{2g+r}+(2g+r)\mathsf{Root}\left(k_0,n,n^{([V:U](2g+r))^2}\right)\right)\] elementary operations.
\begin{proof} Computing $w\nt$ requires $[V:U](2g+r)$ computations of $n^{\text{th}}$ roots in $k$. Computing $I\nt$ requires $[V:U]n^{2g+r}$ function evaluations.
\end{proof}
\end{lem}

\subsection{Computation of $\RG$}

Let $k_0$ be a perfect field of characteristic prime to $n$, and $k$ be an algebraic closure of $k_0$. Let $X_0$ be an integral curve over $k_0$ with ordinary singularities, and $X=X_0\times_{k_0}k$. Consider a complex $\F_0^\bullet=[\F_0^0\to\dots\to\F_0^t]$ of constructible sheaves of $\ZZ/n\ZZ$-modules on $X_0$, and set $\F^\bullet\coloneqq (\F_0)^\bullet|_X$. Let $U$ be a smooth open affine subscheme of $X$ such that $\F^\bullet|_U$ is a complex of locally constant sheaves. Let $r$ be the number of points of $X-U$. Let $V\to U$ be an étale Galois cover such that, for each integer $i\in \{0\dots t\}$, the sheaf $\F^i|_U$ is constant. Denote by $G$ the automorphism group of $V\to U$. The following algorithm computes the cohomology complex $\RG(X,\F^\bullet)$ described in \Cref{th:RGamma}.

\begin{algorithm}[H]\label{alg:RGamma}
\SetAlgoLined  
\caption{\textsc{RGamma}}
\KwData{Integral curve $X$ over alg. closed field $k$ \\
Integer $n$ invertible in $k$ \\
Constructible sheaf complex $\F^\bullet$ on $X$ as described in section \ref{subsec:rep}:\\
affine open $U\subset X$ where each $\F^i$ is locally constant,\\
Galois cover $V\to U$ which trivialises $\LL^\bullet\coloneqq \F^\bullet|_U$,\\
Galois group $\Aut(V|U)$ and inertia subgroups $I_z\subset G$ for $z\in X-U$,\\
generic fibres $M^i$ of $\F^i$ with action of $\Aut(V|U)$,\\
fibres $\F_z^i$ for $z\in X-U$,\\
adjunction units $\phi_z^i\colon \F_z^i\to (M^i)^{I_z}$.
}
\KwResult{Complex of $\Lambda$-modules representing $\RG(X,\F)$}
\hrulefill \\
Compute $\Aut(\Vnt|U)$ using Algorithm \ref{alg:X2} \\
Compute inertia subgroups $I_z'\subset \Aut(\Vnt|U)$ using Algorithm \ref{alg:inertia}\\
Using linear algebra, compute $\Homcr(\Aut(\Vnt|U),M^i)$ and $\Homcr(I_z',M^i)$ for $z\in X-U$\\
Compute the morphism $\Psi\colon \RG(X,i_\star i^\star\F^\bullet)\oplus\RG(X,j_\star \LL^\bullet)\to \RG(X,i_\star i^\star j_\star\LL^\bullet)$ of \Cref{th:RGamma}\\
\Return $\cone(\Psi)[-1]$
\end{algorithm}

\begin{theorem} Let $m$ be an integer such that $M$ and the fibres $\F_z^i$, $z\in Z$, $i\in\{0\dots t\}$ are given by at most $m$ generators. Denote by $d$ the degree of an ordinary plane model of $V$. Algorithm \ref{alg:RGamma} computes a complex of $\Lambda$-modules representing $\RG(X,\F)$ in \[\adjustbox{max width=\textwidth}{$\mathsf{H1const}(k_0,n,d,[V:U](2g+r))+\Poly\left((n^{[V:U](2g+r))^2},m,t\right)+[V:U](2g+r)\mathsf{Root}\left(k_0,n,n^{([V:U](2g+r))^2}\right)$}\]
elementary operations. When $k_0=\FF_q$, this number is bounded by \[ \Poly\left(n^{([V:U](2g+r))^2},n^d,m,\log q,t \right).\]
\begin{proof} 
This is just putting together the complexities of the previous algorithms, taking into account that the computation of modules of crossed homomorphisms is done using linear algebra over $\Lambda$.
In order to bound the number in the case of a finite field, we use the complexity of Huang and Ierardi's algorithm to compute $\HH^0(V,\mu_n)$.
\end{proof}
\end{theorem}

\begin{rk} The only existing algorithm computing $\HH^1(X,\F)$ when $\F$ is locally constant is Jin's algorithm; its complexity is exponential in $|M|^{\log|M|}$ \cite[Th. 1.2]{jinbi_jin}, to which the complexity of our algorithm compares favourably.
\end{rk}

\subsection{Improving complexity when $k_0$ is finite}\label{subsec:ffieldcoh}

Here we consider the case where $k_0=\FF_q$ is a finite field, $X_0$ is a smooth curve over $k_0$, and $\F_0$ is a constructible sheaf of $\FF_\ell$-vector spaces on $X_0$, where $\ell$ is a prime not dividing $q$. Let $U_0$ be an open subset of $X_0$ on which $\F_0$ is locally constant. Denote by $V_0\to U_0$ an étale Galois cover such that $\F_0|_{V_0}$ is constant with fibre $M$. For simplicity, we assume $V_0$ to be geometrically connected; if it were not, the field extension $\FF_Q$ defined below should be replaced by its compositum with the field extension defined in \Cref{subsec:X2gal}.
Write $m=\dim_{\FF_\ell}(M)$. Denote by $k=\overline{\FF_q}$ an algebraic closure of $k_0$, and by $X, U, V$ the base changes of $X_0, U_0, V_0$ to $k$. Denote by $g_X$ the genus of $X$ and by $r$ the number of points of $X-U$. 

\begin{lem} Set $D=|\GL_{m(2g_X+r)}(\FF_\ell)|$ and $Q=q^D$. Consider a basis $(D_1,f_1),\dots, (D_s,f_s)$ of the $\FF_\ell$-vector space $\HH^1(V,\mu_\ell)(\FF_Q)$, that is, the subspace of elements of $\HH^1(V,\mu_\ell)$ invariant under the action of $\Gal(k|\FF_Q)$. Denote by $V\Qlt$ the étale Galois cover of $V$ defined by the function field extension $k(V)(\sqrt[\ell]{f_1},\dots,\sqrt[\ell]{f_s})$. The map $\HH^1(U,\F|_U)\to \HH^1(V\Qlt,\F|_{V\Qlt})$ is trivial.
\end{lem}
\begin{proof}
We know that the action of $\Gk$ factors through a finite quotient $\Gal(k_1|k_0)$, where $[k_1:k_0]$ divides $\Aut_{\FF_\ell}(\HH^1(U,\F))$. Recall the prime-to-$p$ fundamental group of $U$ is generated by at most $2g+r$ elements.
Now $\HH^1(U,\F)$ is a quotient of $\Homcr(\pi_1(U)^{(p')},M)$; its dimension as an $\FF_\ell$-vector space is bounded above by $m(2g_X+r)$. Therefore, $\Aut_{\FF_\ell}(\HH^1(U,\F))$ injects into $\GL_{m(2g_X+r)}(\FF_\ell)$.  Note that since $\ell$ divides $Q-1$, the field $\FF_Q$ contains a primitive $\ell^{\rm th}$ root of unity $\zeta$, and the isomorphism $\HH^1(V,\FF_\ell)\to \HH^1(V,\mu_\ell)$ defined by $1\mapsto \zeta$ is $\Gal(k|\FF_Q)$-equivariant. 
Since the $U$-automorphisms of $V$ are defined over $\FF_q$, the set $\HH^1(V,\mu_\ell)(\FF_Q)$ is stable under the action of $\Aut(V|U)$, and $V'\to U$ is still Galois. 
The construction of $V\Qlt$ ensures that $\HH^1(V,\FF_\ell)(\FF_Q)\to\HH^1(V\Qlt,\FF_\ell)$ is trivial. The $m$ copies of $\F|_V\simeq\FF_\ell^m$ being stable under the action of $\Gk$, the map \[\HH^1(U,\F)\to \HH^1(V\Qlt,\F)\xrightarrow{\sim} \HH^1(V\Qlt,\FF_\ell)^m\] factors through $\HH^1(V,\FF_\ell)(\FF_Q)^m$, and is also trivial.
\end{proof}

Hence, we may use $V\Qlt$ instead of $V^{\langle \ell \rangle}$ in Algorithm \ref{alg:RGamma}. Note that $D\leqslant \ell^{(m(2g+r))^2}$. 

\paragraph{Computing $\HH^1(V,\FF_\ell)(\FF_Q)$} This group is isomorphic to $\HH^1(V,\mu_\ell)(\FF_Q)$.
Denote by $\bar V$ the smooth projective curve containing $V$, and by $J_{\bar V}$ its Jacobian. Consider a basis $(D_i,f_i)_{1\leqslant i\leqslant s}$ of $\HH^1(V,\mu_\ell)(\FF_Q)$. Here, the $D_i$ are divisors on $V$ such that $\ell D_i=\div_V(f_i)$. The proof of \Cref{lem:H1div} tells us that we may assume some of these pairs to form a basis of $J_{\bar V}[\ell](\FF_Q)$. For the remaining ones, the elements $M_i\coloneqq \ell D_i-\div_{\bar V}(f_i)$ form a basis of the space of $\Gal(k|\FF_Q)$-invariant elements of the kernel of the following map. \[ \begin{array}{rcl}\HH^0(\bar V - V,\FF_\ell)&\longrightarrow& \FF_\ell \\ (\lambda_P)_{P\in\bar V- V} &\longmapsto& \sum_P\lambda_P\end{array}\] Finding the $D_i$ amounts to dividing the $M_i$ by $\ell$ in $J_{\bar V}$. Here is how to do this.
Any element of $J(\FF_Q)$ has order dividing $Q-1$. Write $Q-1=\ell^\alpha s$ with $s$ prime to $\ell$. Then $sM_i\in J[\ell^\alpha](\FF_Q)$, and we may find in $J[\ell^{\alpha+1}](\FF_Q)$ an element $E_i$ such that $\ell E_i=sM_i$: to do this, compute $J[\ell^{\alpha+1}](\FF_Q)$ and use linear algebra. Actually, given the definition of $\alpha$, $J[\ell^{\alpha+1}](\FF_Q)=J[\ell^{\alpha}](\FF_Q)$. Considering integers $u,v$ such that $u\ell+vs=1$, we have  $M_i=\ell \cdot (uM_i+vE_i)$.\\

\paragraph{Complexity} Since $Q=q^D$ where $D=|\GL_{m(2g_X+r)}(\FF_\ell)|$, the integers $\alpha$ and $s$ can be computed easily.
As soon as $q^{\ell-1}\neq 1\mod \ell^2$, we know $q^{\ell-1}$ has order $\ell^{\alpha-1}$ in $(\ZZ/\ell^\alpha \ZZ)^\times$ and \[ \alpha-1\leqslant v_\ell(D)=\frac{(m(2g_X+r))(m(2g_X+r)-1)}{2}\]
which is polynomial in $m,g_X,r$. The complexity of computing $\HH^1(V,\mu_\ell)(\FF_Q)$ is dominated by the computation of $J[\ell^\alpha](\FF_Q)$. Couveignes' algorithm computes $J[\ell^\alpha](\FF_Q)$ in time polynomial in $\ell^\alpha$, the genus of $\bar V$ and $\log(Q)$, assuming the characteristic polynomial of the $Q$-Frobenius on $V$ is known.

\subsection{Potential application: point counting on surfaces}

Let $X_0$ be a smooth projective surface over a finite field $k_0=\FF_q$. Denote by $k$ an algebraic closure of $k_0$, and set $X=X_0\times_{k_0}k$. Consider the problem of computing $|X(k_0)|$. The usual approach, as in Schoof's algorithm, is to compute this number modulo $\ell$ for enough primes $\ell$ up to $O(\log q)$.

The Lefschetz theorem reduces this question to computing the trace of the Frobenius on $\HH^i(X,\FF_\ell)$. The classical way of computing these groups, as in \cite[§V.3]{milneEC}, is by using a Lefschetz pencil, which yields a fibration $\pi\colon \tilde X\to \PP^1$, where $\tilde X$ is a blowup of $X$ at a finite number of points. Edixhoven conjectured in \cite[Epilogue]{edixcouv} that this strategy might allow us to compute $|X(k_0)|$ in time polynomial in $\log(q)$. Here is where we stand on this conjecture.
The sheaf $\F\coloneqq \R^1\pi_\star\FF_\ell$ is a constructible sheaf on the projective line.  It is locally constant on the open subset $U$ of $\PP^1$ over which the fibres of $\pi$ are smooth curves. For $z\in \PP^1-U$, we know how to compute $\RG(X_z,\FF_\ell)$. Moreover, given an explicit description of $\F$, we know how to compute $\RG(\PP^1,\F)$. Denote by $X_\bareta$ the generic fibre of $\pi$, and by $g_\bareta$ its genus. A trivialising cover $V\to U$ of $\F|_U$ is given by the normalisation of $U$ in an extension of $K$ of $k(t)$ over which $\Pic(X_{\bareta})[\ell]$ is defined. The following data helps us estimate the complexity we need for the different steps of the algorithms: \begin{itemize}
\item the degree $[K:k(t)]$ is smaller than $\ell^{4g_\bareta^2}$;
\item the number $r=|\PP^1-U|$ only depends on $\pi$ and not on $\ell$;
\item the genus $g_V$ of $V$ is bounded by $r\ell^{4g_\bareta^2}$.
\end{itemize} 
Computing the whole cover $V\lt$ would be too costly. However, as suggested in section \ref{subsec:ffieldcoh}, it is sufficient to compute a subcover $V\Qlt$ of $V\lt$ using only elements of $\HH^1(V,\mu_\ell)$ defined over a degree $D=O(\ell^{4g_\bareta^2 r^2})$ extension $\FF_Q$ of $\FF_q$. Hence, if we could compute: \begin{itemize}
\item $\HH^1(X_\bareta,\mu_\ell)$ in time polynomial in $\ell$ and $\log(q)$,
\item $\HH^1(V,\mu_\ell)(\FF_Q)$ in time polynomial in $\ell,\log Q$ and the genus of $V$,
\end{itemize}
then we should be able to compute the $\HH^i(X,\F)$ with their Frobenius action in time $\Poly(\ell,\log q)$. Mascot recently described an algorithm to deal with the first item of the list \cite[Alg. 2.2]{mascot_2023} using $p$-adic approximation; however, parts of his method are not yet rigorous \cite[Rk. 4.3]{mascot_2023}.\\

For the moment, this is nothing more than wishful thinking: all existing algorithms to compute $\HH^1(V,\mu_\ell)$, even for projective curves, have complexity exponential either in $\log(q)$ or in the genus of $V$. However, there is some hope.   Harvey's algorithm \cite[Th. 1]{harvey}, which computes the zeta function of hyperelliptic curves, reaches an average polynomial-time complexity. Combined with Couveignes' algorithm and \Cref{subsec:ffieldcoh}, this allows for an average polynomial-time complexity for the computation of $\HH^1(V,\mu_\ell)(\FF_Q)$ in the case where $V$ is an open subset of a hyperelliptic curve.
\section{First example: sheaves on subschemes of $\PP^1$}\label{sec:examples1}

\subsection{The cover}

Take $n=2$. Let $k_0$ be a field of odd characteristic in which $-1$ is not a square. Consider the degree 2 (ramified) Galois cover 
\[\bar f\colon \begin{array}[t]{rcl} \PP^1&\longrightarrow& \PP^1 \\ y&\longmapsto& y^2\end{array}\]
whose automorphism group is generated by $\tau\colon y\mapsto -y$. Set $U=\PP^1-\{ 0,1,\infty\}$ and $V=f^{-1}U=\PP^1-\{ 0,\pm 1,\infty \}$, and consider the étale cover $f\colon V\to U$ induced by $\bar f$.

\paragraph{Computation of $\Vnt$} The group $\HH^1(V,\mu_2)\simeq\Lambda^3$ is generated by the divisor-function pairs $(0-\infty,y),(1-\infty,y-1),(-1-\infty,y+1)$. The cover $\Vnt\to V$ with group $\HH^1(V,\Lambda)^\vee$ corresponds to the field extension $k(\sqrt{y},\sqrt{y-1},\sqrt{y+1})/k(y)$. The corresponding cover of $\bar V=\PP^1$ is the map 
\[ \begin{array}{rcc} \Proj k[z_0,z_1,z_2,z_3]/(z_1^2-(z_0^2-z_3^2),z_2^2-(z_0^2+z_3^2)) &\longrightarrow& \Proj k[y_0,y_1] \\ (z_0:z_1:z_2:z_3)&\longmapsto& (z_0^2:z_3^2)\end{array}\]
which is ramified above $0,\pm 1,\infty$. 

\paragraph{Computation of $\Aut(\Vnt|U)$} The automorphism group $G\coloneqq \Aut(\Vnt|U)$ has order 16; in order to compute all of its elements, it suffices to compute a preimage of $\tau$ in $\Aut(\Vnt|U)$. Such a preimage is given by $\gamma \colon (z_0:z_1:z_2:z_3)\mapsto (\sqrt{-1}z_0:\sqrt{-1}z_2:\sqrt{-1}z_1:z_3)$. Let \[
\begin{array}{rcl} \sigma_1\colon (z_0:z_1:z_2:z_3)\mapsto (-z_0:z_1:z_2:z_3)\\ \sigma_2\colon (z_0:z_1:z_2:z_3)\mapsto (z_0:-z_1:z_2:z_3) \\ \sigma_3\colon (z_0:z_1:z_2:z_3)\mapsto (z_0:z_1:-z_2:z_3)\end{array} \]  be the obvious generators of $\Aut(\Vnt|V)\triangleleft G$. Then $\gamma\sigma_2=\sigma_3\gamma$ and $\gamma\sigma_3=\sigma_2\gamma$, which implies that $\langle \sigma_2,\sigma_3\rangle$ is normal in $G$. It is easy to check that the composite map \[\langle \gamma \rangle \to G/\langle \sigma_2,\sigma_3\rangle\] is an isomorphism; therefore, \[ G=\langle \sigma_2,\sigma_3 \rangle \rtimes \langle \gamma\rangle.\]

\subsection{Cohomology of a locally constant sheaf}

The sheaf $\F\coloneqq f_\star\Lambda$ is locally constant on $U$, trivialised by the cover $f\colon V\to U$ since $f^\star f_\star \Lambda\simeq \Lambda^2$. It corresponds to the $\Aut(V|U)$-module $\Lambda^2$, where the non-trivial element of $\Aut(V|U)$ exchanges the two copies of $\Lambda$.  Since $f$ is finite, $\R f_\star \Lambda=(f_\star\Lambda)[0]$ and there is a canonical isomorphism \[\RG(U,f_\star \Lambda)=\RG(V,\Lambda).\] We therefore expect to find \[\HH^1(U,\F)=\HH^1(V,\Lambda)\simeq\Lambda^3.\]

\paragraph{Computing $\RG(U,\F)$} We know that $\RG(U,\F)$ is represented by the following two-term complex: 
\[ \Lambda^2 \to \Homcr(G,\Lambda^2). \]
A crossed homomorphism $f\colon G\to \Lambda^2$ is determined by the images of $\sigma_1,\sigma_2,\sigma_3,\gamma$. Using the relations $\gamma\sigma_1=\sigma_1\gamma, \gamma\sigma_2=\sigma_3\gamma$ and $\gamma^2=\sigma_1\sigma_2\sigma_3$, we see that such a map is uniquely determined by a tuple $(a,a_1,a_2,a_3)\in \Lambda^4$; the corresponding map $f$ is defined by $f(\sigma_1)=(a_1,a_1)$, $f(\sigma_2)=(a_2,a_3)$, $f(\sigma_3)=(a_3,a_2)$ and $f(\gamma)=(a,a+a_1+a_2+a_3)$. The principal crossed homomorphisms correspond to $(0,0,0,0)$ and $(1,0,0,0)$. Hence the complex above is isomorphic to 
\[ \begin{array}{rcl}\Lambda^2 &\longrightarrow&\Lambda^4 \\ (a,b) & \longmapsto &(a+b,0,0,0) \end{array} \] and its cohomology groups are $\Lambda$ and $\Lambda^3$, as expected. 

\paragraph{Computing the Galois action} The action of $\mathfrak{G}_0=\Gal(k|k_0)$ on $\Aut(Y|X)$ clearly factors through the quotient $\Gal(k_0(\sqrt{-1})|k_0)$. The latter group is generated by $\phi\colon \sqrt{-1}\mapsto -\sqrt{-1}$. The automorphism $\phi$ acts trivially on $\sigma_1,\sigma_2,\sigma_3$, and \[\phi\cdot \gamma=\sigma_1\sigma_2\sigma_3\gamma\colon (z_0 : z_1 : z_2 :z_3)\mapsto (-\sqrt{-1}z_0:-\sqrt{-1}z_1:-\sqrt{-1}z_2:z_3).\] The action of $\phi$ on $\Lambda^2$ is trivial, and its action on $\Homcr(G,\Lambda^2)\simeq \Lambda^4$ is $(\phi\cdot f)(x)=\phi f(\phi^{-1}x)=f(\phi^{-1}x)$. Explicitly, since $\phi\cdot\gamma=\sigma_1\sigma_2\sigma_3\gamma$, this action is given by \[\phi\cdot (a,a_1,a_2,a_3)=(a+a_1+a_2+a_3,a_1,a_2,a_3).\]

\subsection{Ramification}

The following illustrates \Cref{subsec:ramgrp} and provides a few results that will be used in the next example.
\paragraph{Ramification} Let $Z,W,W'$ denote the sets of points at infinity of $U,V,\Vnt$ respectively. The following table gives an overview of the situation.

\begin{table}[H]
\resizebox{\textwidth}{!}
{\begin{tabular}{|l|clll|clllclll|clll|}
\hline
Points in $Z$                                                                                  & \multicolumn{4}{c|}{$0$}                                                                                  & \multicolumn{8}{c|}{$1$}                                                                                                                                                                                                     & \multicolumn{4}{c|}{$\infty$}                                                                       \\ \hline
\begin{tabular}[c]{@{}l@{}}Preimages in $W$ \\ Ramification \end{tabular}  & \multicolumn{4}{c|}{\begin{tabular}[c]{@{}c@{}}$0$\\ index $2$\end{tabular}}                                    & \multicolumn{4}{c|}{\begin{tabular}[c]{@{}c@{}}$-1$\\ index $1$\end{tabular}}                                           & \multicolumn{4}{c|}{\begin{tabular}[c]{@{}c@{}}$1$\\ index $1$\end{tabular}}                                   & \multicolumn{4}{c|}{\begin{tabular}[c]{@{}c@{}}$\infty$\\ index $2$\end{tabular}}                         \\ \hline
\begin{tabular}[c]{@{}l@{}}Preimages in $W'$ \\ Ramification \end{tabular} & \multicolumn{4}{c|}{\begin{tabular}[c]{@{}c@{}}4 points \\index 4\end{tabular}} & \multicolumn{4}{c|}{\begin{tabular}[c]{@{}c@{}}4 points \\ index 2 \end{tabular}} & \multicolumn{4}{c|}{\begin{tabular}[c]{@{}c@{}}4 points \\index 2 \end{tabular}} & \multicolumn{4}{c|}{\begin{tabular}[c]{@{}c@{}}4 points \\ index 4 \end{tabular}} \\ \hline
\begin{tabular}[c]{@{}l@{}}A preimage in $W'$ \\ Its inertia group\end{tabular} & \multicolumn{4}{c|}{\begin{tabular}[c]{@{}c@{}}$P_0=(0,\sqrt{-1},1)$ \\ $\langle \gamma\sigma_2\rangle\simeq \mu_4(k)$\end{tabular}} & \multicolumn{4}{c|}{\begin{tabular}[c]{@{}c@{}}$P_{-1}=(\sqrt{-1},\sqrt{-2},0)$
 \\ $\langle \sigma_3\rangle\simeq \mu_2(k)$ \end{tabular}} & \multicolumn{4}{c|}{\begin{tabular}[c]{@{}c@{}}$P_1=(1,0,\sqrt{2})$
 \\ $\langle \sigma_2\rangle\simeq \mu_2(k)$ \end{tabular}} & \multicolumn{4}{c|}{\begin{tabular}[c]{@{}c@{}}$P_\infty=(1:1:1:0)$
 \\ $\langle \gamma\rangle\simeq \mu_4(k)$ \end{tabular}} \\ \hline
\end{tabular}}
\end{table}

The canonical isomorphism $I_{P_0}\to \mu_4(k)$ can be described explicitly as follows. The function $y$ is a uniformiser of $\Vnt$ at $P_0=(0,\sqrt{-1},1)$. The orbit of $y$ under the action of $I_{P_0}=\langle \gamma\sigma_2\rangle$ is $\{ \pm y,\pm \sqrt{-1}y\}$. Hence the set $\{ \frac{\sigma(y)}{y}(P_0)\mid\sigma \in I_{P_0}\}$ is exactly $\mu_4(k)$, and the isomorphism $I_{P_0}$ to $\mu_4(k)$ sends an element $\sigma\in I_{P_0}$ to $\frac{\sigma(y)}{y}(P_0)$.

The generator $\sqrt{-1}$ of $\mu_4(k)$ exchanges the two copies of $\Lambda$ in $M=\Lambda^2$. The $\Lambda$-module of crossed homomorphisms $\mu_4(k)\to M$ is isomorphic to $\Lambda^2$, and $\tau_{\leqslant 1}\RG(I_{P_0},M)$ is represented by the following complex. 

\[ \begin{array}{rcl} \Lambda^2 &\to& \Lambda^2\\
(a,b)&\mapsto& (a+b,a+b) \end{array} \] 
The group $\F_0=\HH^0(I_{P_0},M)$ is generated by $(1,1)$, and $\HH^2_0(X,j_\star\F)=\HH^1(I_{P_0},M)$ is generated by the class of $(0,1)$. The compuation of the complex $\tau_{\leqslant 1}\RG(I_{P_{\infty}},M)$ is done in the same way and yields the same result. 
The group $I_{P_1}$ is canonically isomorphic to $\mu_2(k)$, and acts trivially on $M$. Therefore, $\tau_{\leqslant 1}\RG(I_{P_1},M)$ is represented by the following complex.
\[ \begin{array}{rcl}\Lambda^2 &\to& \Lambda^2 \\(a,b)&\mapsto &0 \end{array} \]
The computation for $P_{-1}$ is the same and yields the same result.

\subsection{Cohomology of a constructible sheaf}

We still consider a field $k_0$ of odd characteristic, where $-1$ is not a square. Define $k_1=k_0(\sqrt{-1})$.
Let us now consider the ramified cover of projective curves $\bar f\colon\PP^1\to \PP^1$ defined by $f$, and the sheaf $\F\coloneqq \bar f_\star\Lambda$ on $\PP^1$. Since $\bar f$ is unramified outside $0,\infty$, the sheaf $\F$ is locally constant on the open subset $\GG_m=\PP^1-\{ 0,\infty\}$, and $\LL\coloneqq \F|_{\GG_m}$ is trivialised by $\bar f|_{\GG_m}\colon \GG_m\to \GG_m, y\mapsto y^2$. The cover $\GG_m\nt$ is also $\GG_m\to\GG_m, z\mapsto z^2$, and the composite \[ \GG_m\nt\to \GG_m\xrightarrow{f} \GG_m \] is given by $z\mapsto z^4$. Its automorphism group is $G=\langle \gamma\colon z\mapsto \sqrt{-1}z\rangle\simeq\ZZ/4\ZZ$, and the inertia subgroups at $0$ and $\infty$ are both equal to $G$. Denote by $j$ the inclusion $\GG_m\to\PP^1$. We have $(j_\star\LL)_0=\Lambda$ and $(j_\star\LL)_{\infty}=\Lambda$. 
The adjunction units $\F_0\to (j_\star \mathscr{L})_0$ and $\F_\infty\to (j_\star \mathscr{L})_\infty$ are the identity maps $\Lambda\to \Lambda$. 

\paragraph{Computing $\RG(\GG_m,\LL)$} 
The crossed homomorphisms $G\to M=\LL_{\bareta}\simeq\Lambda^2$ are uniquely determined by the image of $(a,b)\in\Lambda^2$ under $\gamma$. The usual cochain complex representing $\tau_{\leqslant 1}\RG(G,M)=\tau_{\leqslant 1}\RG(\GG_m,\LL)$ is the following.
\[ \begin{array}{rcl}
\Lambda^2 &\longrightarrow& \Homcr(G,\Lambda^2)\\
(a,b)&\longmapsto& \left[ \gamma \mapsto (a+b,a+b)\right]
\end{array}\]
Therefore $\HH^1(G,M)$ is isomorphic to $\Lambda$, and the kernel of the map $\Lambda^2\to \HH^1(G,M)$ sending a crossed homomorphism to its cohomology class is $\langle (1,1)\rangle$ ; this map can be rewritten as \[ \begin{array}{rcl}\Lambda^2&\longrightarrow& \Lambda \\ (a,b)&\longmapsto& a+b.\end{array}\]

\paragraph{Computing $\RG(X,j_\star \mathscr{L})$} 
The element $\RG(X,j_\star \mathscr{L})[1]\in \DD^b_c(X,\Lambda)$ is the cone of \[\tau_{\leqslant 1}\RG(G,M)\to \HH^1(I_0,M)[-1]\oplus \HH^1(I_\infty,M)[-1]=\Lambda^2[-1].\] Therefore, $\RG(X,j_\star \mathscr{L})$ is represented by the complex 
\[ \Lambda^2\longrightarrow \Lambda^2\longrightarrow \Lambda^2 \]
where both morphisms are given by $(a,b)\mapsto (a+b,a+b)$. 

\paragraph{Computing $\RG(X,\F)$} 
Let us now turn to the computation of the map \[ \RG(X,j_\star \mathscr{L})\oplus \RG(Z,i^\star\F) \to \RG(Z,i^\star j_\star \mathscr{L}).\]
On the one hand, $\RG(Z,i^\star\F)=\HH^0(Z,i^\star\F)[0]=\F_0[0]\oplus \F_\infty[0]$. On the other hand, $\RG(Z,i^\star j_\star \mathscr{L})$ is represented by the complex
\[ \begin{tikzcd}
\Lambda^2\oplus\Lambda^2 \arrow[r,"{\alpha'}"]& \Lambda^2\oplus\Lambda^2 \arrow[r,"{\beta'}"]& \Lambda^2 \end{tikzcd}\]
where the arrows are given by $\alpha'\colon (a,b,c,d)\mapsto (a+b,a+b,c+d,c+d)$ and $\beta'\colon (a,b,c,d)\mapsto(a+b,c+d)$.
Hence the map we were looking for is
\[
\begin{tikzcd}
\Lambda^4 \arrow[r,"{\alpha}"]\arrow[d,"u"] & \Lambda^2 \arrow[r,"{\beta}"]\arrow[d,"v"] &\Lambda^2\arrow[d,"{\id}"] \\
\Lambda^4 \arrow[r,"{\alpha'}"] & \Lambda^4\arrow[r,"{\beta'}"] & \Lambda^2
\end{tikzcd}
\]
where, writing the upper left-hand side term as $\F_0\oplus\F_\infty\oplus M$, the arrows are given by \begin{itemize}[label=$\bullet$]
\item $u\colon(a,b,c,d)\mapsto (a+c,a+d,b+c,b+d)$
\item $\alpha \colon (a,b,c,d)\mapsto (c+d,c+d)$
\item $\beta \colon(a,b)\mapsto(a+b,a+b)$
\item $v\colon(a,b)\mapsto (a,b,a,b)$
\item $\alpha'\colon (a,b,c,d)\mapsto (a+b,a+b,c+d,c+d)$
\item $\beta'\colon(a,b,c,d)\mapsto(a+b,c+d)$.
\end{itemize} 
Computing the cone of this morphism and shifting by 1 yields the following complex, which represents $\RG(X,\F)$.
\[
\begin{tikzcd}
\Lambda^4 \arrow[r,"\partial_0"] & \Lambda^6 \arrow[r,"\partial_1"] & \Lambda^6 \arrow[r,"\partial_2"] & \Lambda^2
\end{tikzcd}
\]
\begin{itemize}[label=$\bullet$]
\item $\partial_0\colon(a,b,c,d)\mapsto (c+d,c+d,a+c,b+c,a+d,b+d)$
\item $\partial_1 \colon(a,b,c,d,e,f)\mapsto (a+b,a+b,a+c+d,b+c+d,a+e+f,b+e+f)$
\item $\partial_2\colon(a,b,c,d,e,f)\mapsto (a+c+d,b+e+f)$. 
\end{itemize}
The cohomology groups of this complex are $\HH^0=\langle (1,1,1,1)\rangle$, $\HH^1=0$ and $\HH^2=\langle \overline{(1,0,1,0,0,0)}\rangle\simeq\Lambda$. This result was to be expected: we have computed the cohomology of $(\PP^1\xrightarrow{x\mapsto x^2}\PP^1)_\star\Lambda$, which is the cohomology of the constant sheaf $\Lambda$ on $\PP^1$.

\paragraph{Computing the Galois action} The action of $\Gk$ on $\RG(X,\F)$ factors through the quotient $\Gal(k_0(\sqrt{-1})|k_0)$ of $\Gk$. Denote by $\sigma\colon \sqrt{-1}\mapsto -\sqrt{-1}$ the nontrivial element of $\Gal(k_0(\sqrt{-1})|k_0)$. The group $\Gk$ acts on $G$ by $\sigma\cdot \gamma=\gamma^3$, and trivially on $M$. Its action on $\tau_{\leqslant 1}\RG(G,M)=\Lambda^2\to \Lambda^2$ is trivial on the first term, and $(a,b)\mapsto (b,a)$ on the second. In particular, it acts trivially on $\HH^1(G,M)=\Lambda$.
The action of $\sigma$ on $\F_0,\F_\infty$ is also trivial. Hence the action of $\Gk$ on the complex
\[ \Lambda^4\to \Lambda^6\to\Lambda^6\to \Lambda^2 \]
representing $\RG(X,\F)$ is trivial on the first and last terms, \[\sigma\cdot (a,b,c,d,e,f)= (b,a,c,d,e,f)\] on the second term, and \[\sigma\cdot (a,b,c,d,e,f)= (a,b,d,c,f,e)\] on the third term. In particular,  $\Gk$ acts trivially on $\HH^0(X,\Lambda)$ and $\HH^2(X,\Lambda)$, as expected.

\section{Second example: sheaves on subschemes of an elliptic curve}\label{sec:examples2}

The examples in this section illustrate a non-trivial case in which we can easily compute the cohomology of sheaves using (almost) only functions that are already available in current computer algebra systems: when the locally constant part of the sheaf is trivialised by a subscheme of a hyperelliptic curve. 

\subsection{The cover}

Consider the finite field $k_0=\FF_{11}$, and the integer $n=2$ invertible in $k_0$. Consider the field extension $\FF_{121}=\FF_{11}(a)$, where $a$ generates the cyclic group $\FF_{121}^\times$ and $a^2+7a+2=0$. Denote by $\bar E$ the elliptic curve over $k=\overline{\FF_{11}}$ defined by the affine Weierstrass equation $y^2=(x-1)(x-2)(x-3)$. 
Let $\bar C$ be the genus 2 curve over $k$ given by the affine equation $y^2=(x^2-1)(x^2-2)(x^2-3)$. The curve $\bar C$ has two points at infinity $\infty_+,\infty_-$ which do not lie on the affine open defined by this equation. Consider the degree 2 cover $\bar f\colon \bar C\to \bar E$ given by $(x,y)\mapsto (x^2,y)$. It is ramified at the affine points $P=(0,4)$ and $Q=(0,7)$ of $\bar C$, whose images in $E$ are respectively $(0,4)$ and $(0,7)$. Denote by $C=\bar C-\{ P,Q\}$ and $E=\bar E-\{ \bar f(P),\bar f(Q)\}$ the affine curves obtained from $\bar C$ and $\bar E$ by removing the ramification locus. Denote by $f\colon C\to E$ the étale Galois cover induce by $\bar f$. Both curves $C$ and $E$ are obtained by base change from curves $C_0$ and $E_0$ defined over $k_0=\FF_{11}$.\\

Computing the Galois group of $C\nt\to E$ will allow us to determine the cohomology of any locally constant sheaf on $E$ trivialised by $f$. First, we need to compute $\HH^1(C,\mu_2)$. This is where we need to cheat a little since no algorithm performing this computation for a general curve has been implemented yet (see \Cref{subsec:exalg} for existing algorithms); fortunately, $\bar C$ being a genus two curve, we have other means of finding a generating set for this group. 

\paragraph{Computing $\HH^1(C,\mu_2)$} Denote by
\[P_1^\pm=(\pm 1,0), \qquad P_2^\pm=(\pm a^6,0),\qquad P_3^\pm =(\pm 5,0)\]
the points of $C$ with $y$-coordinate 0. A basis of $\Jac(\bar C)[2]$ is given by the classes of the divisors \[ D_1\coloneqq P_1^+-P_1^-,\qquad D_2\coloneqq P_2^+-P_2^-,\qquad D_3\coloneqq P_2^+-P_3^+,\qquad D_4\coloneqq P_1^+-P_3^-.\]
The rational functions \[ f_1\coloneqq \frac{x-1}{x+1},\qquad f_2\coloneqq \frac{x-a^6}{x+a^6},\qquad f_3\coloneqq\frac{x-a^6}{x-5},\qquad f_4\coloneqq\frac{x-1}{x+5}\]
all satisfy $2D_i=\div(f_i)$. Denote by $D_5$ the divisor $P-Q$, which is linearly equivalent to two times the divisor \[ \bar D_5\coloneqq (a^{41},a^{29})+(-a^{41},a^{29})-(\infty_++\infty_-).\] We found this divisor $\bar D_5$ using a brute-force search on $\Jac(\bar C)(\FF_{121})$; for hyperelliptic curves such as $\bar C$, this can also be done using division polynomials (see e.g. \cite[Theorem C]{eid_division}). In the particular case where $n=2$, division by 2 has even been explicitly described by Zarhin \cite[Th. 3.2]{zarhin_div}. The divisor of the rational function
 \[ f_5=\frac{y+a^8x^2+7}{x}\]
is $2\bar D_5-D_5$. Therefore, an $\FF_2$-basis of $\HH^1(C,\mu_2)$ is given by \[(D_1,f_1),\dots,(D_4,f_4),(\bar D_5,f_5).\]

\paragraph{Computing $\Aut(C\nt|E)$} The cover $C\nt\to C$ with group $\HH^1(C,\Lambda)^\vee$ is defined by its function field $k(C)(z_1,\dots,z_5)$ where $z_i^2=f_i$. Recall that we never need to compute a smooth model of $C\nt$. For reasons explained in section \ref{subsec:ramex}, we choose to replace $f_5$ with $x^2f_5$, which still yields the same function field.
The group $G=\Aut(C\nt|E)$ has order 64; it contains the normal subgroup $H=\Aut(C\nt|C)\simeq (\ZZ/2\ZZ)^5$ generated by the elements $\gamma_i\colon z_i\mapsto -z_i$. Let us now determine a preimage in $G$ of the generator $\sigma\colon (x,y)\mapsto (-x,y)$ of  $\Aut(C|E)$. First, we compute the divisors $\sigma^\star D_i$.
\[
\begin{array}{rllll}
\sigma^\star D_1 & = & -D_1 & &\\
\sigma^\star D_2 & = & -D_2 & &\\
\sigma^\star D_3 & = & D_1+D_3+\div(h_3) & \text{where} & h_3=\frac{y}{x^3+a^{58}x^2+a^2x+a^{54}}\\
\sigma^\star D_4 & = & D_2+D_4+\div(h_4) & \text{where} & h_4=\frac{y}{x^3+a^{80}x^2+a^{103}x+a^{114}}\\
\sigma^\star D_5 & = & D_5
\end{array}
\]
Note that $\sigma^\star f_3=h_3^2f_1f_3$, $\sigma^\star f_4=h_4^2f_2f_4$ and $\sigma^\star f_5=-f_5$. Since $a^{30}$ is a square root of $-1$ in $\FF_{121}$, the automorphism $\delta\in G$ given by \[\adjustbox{max width=\textwidth}{$ 
x\mapsto -x, \quad y\mapsto y,\quad z_1\mapsto \frac{1}{z_1},\quad z_2\mapsto \frac{1}{z_2},\quad z_3\mapsto h_3(x,y)z_1z_3,\quad z_4\mapsto h_4(x,y)z_2z_4,\quad z_5\mapsto a^{30}z_5$}\]
is a preimage of $\sigma$. In particular, since $h_3(x,y)h_3(-x,y)=h_4(x,y)h_4(-x,y)=-1$, the automorphism $\delta^2$ is given by \[ 
x\mapsto x, \quad y\mapsto y,\quad z_1\mapsto z_1,\quad z_2\mapsto z_2,\quad z_3\mapsto -z_3,\quad z_4\mapsto -z_4,\quad z_5\mapsto -z_5.\] Hence $\delta^2=\gamma_3\gamma_4\gamma_5$ and the order of $\delta$ as an element of $G$ is 4. Furthermore, $\delta\gamma_1=\gamma_1\gamma_3\delta$ and $\delta\gamma_2=\gamma_2\gamma_4\delta$. 
The elements $\gamma_3,\gamma_4,\gamma_5$ generate the center of $G$. Its commutator subgroup is $\langle \gamma_3=[\gamma_1,\delta],\gamma_4=[\gamma_2,\delta]\rangle$. Finally, let us compute the action of $\Gal(\FF_{121}|\FF_{11})=\langle\phi\colon a\mapsto 2a^{-1}\rangle$ on $G$, which can be done very easily since the elements of $G$ are defined by their action on coordinates of points. Predictably, $\phi$ acts trivially on the $\gamma_i$; it also sends $\delta$ to $\gamma_5\delta$.

\subsection{Ramification}\label{subsec:ramex}

Here is how to compute the preimages of the points $P$ and $Q$ in $C\nt$. We have computed its function field as $k(C\nt)=k(C)(z_1^2-f_1,\dots,z_5^2-f_5)$, and would now like to compute an actual affine model of $C\nt$. For each $i\in\{1\dots 5\}$, write $f_i=\frac{g_i}{h_i}$. For $P$ and $Q$ to have easily computable preimages, we can replace $f_i$ with $h_i^2f_i$ when $h_i(P)=0$ or $h_i(Q)=0$. This is only the case for $f_5$, which now reads $x(y+a^8x^2+7)$. The following affine curve is birational to $C\nt$:   
\[ \begin{array}{rl}\Spec k[x,y,z_1,\dots,z_5]/&(y^2-(x^2-1)(x^2-2)(x^2-3),\\
& (x+1)z_1^2-(x-1),\\
& (x+a^6)z_2^2-(x-a^6),\\
& (x-5)z_3^2-(x-a^6),\\
& (x+5)z_4^2-(x-1),\\
& z_5^2-x(y+a^8x^2+7)).
\end{array} \]
The $\frac{1}{2}|\HH^1(C,\mu_2)|=16$ preimages of $P=(0,4)$ in $C\nt$ are the points $(0,4,\pm 1, \pm a^{30}, \pm 3a^3, \pm 3a^{30},0)$. The preimages of $Q=(0,7)$ are the points $(0,7,\pm 1, \pm a^{30}, \pm 3a^3, \pm 3a^{30},0)$.
Choose two preimages $P_{C\nt}=(0,4,1, a^{30}, 3a^3, 3a^{30},0)$ and $Q_{C\nt}=(0,7,1, a^{30}, 3a^3, 3a^{30},0)$ of $P$ and $Q$ in this affine curve birational to $C\nt$, and denote by $P_E=(0,4),Q_E=(0,7)$ their respective images in $E$. The inertia group $I_{P_{C\nt}|P_E}\subset G$ has order $|I_{P|P_E}|\cdot |I_{P_{C\nt}|P}|=2\times 2=4$; it is generated by $\delta$. The same applies to $I_{Q_{C\nt}|Q}=\langle\delta\rangle$.

\subsection{Cohomology of a locally constant sheaf on $E$}

Consider the locally constant sheaf $\F$ on $E$ trivialised by $C$, with generic fibre $M=\Lambda^3$, defined by the representation:
\[ \begin{array}{rcl}
\Aut(C|E)&\longrightarrow& \GL_3(\Lambda) \\
\sigma&\longmapsto& \begin{pmatrix}
0 & 0  & 1\\
0 & 1 & 0 \\
1 & 0 & 0
\end{pmatrix}
\end{array} \]
Computations using Magma yield $\HH^0(E,\F)\simeq\Lambda^2$ and $\HH^1(E,\F)\simeq\Lambda^8$. More precisely, the 1-cocycles $c_1,\dots,c_8$ below form a basis of $\HH^1(G,M)$; the null-cohomologous cocycle $c'$ is the image of the 0-cocycle $(1~0~0)\in M$.
\begin{center}
\begin{tabular}{|c|c|c|c|c|c|c|}\hline
Cocycle $c$ & $c(\gamma_1)$ & $c(\gamma_2)$ & $c(\gamma_3)$ & $c(\gamma_4)$ & $c(\gamma_5)$ & $c(\delta)$ \\ \hline
$c_1$ & $(1~0~0)$ & $(0~0~0)$ & $(1~0~1)$  & $(0~0~0)$ & $(0~0~0)$ & $(0~0~1)$\\ \hline
$c_2$ & $(0~1~0)$ & $(0~0~0)$ & $(0~0~0)$ & $(0~0~0)$ & $(0~0~0)$ & $(0~0~0)$\\ \hline
$c_3$ & $(0~0~1)$ & $(0~0~0)$ & $(1~0~1)$  & $(0~0~0)$ & $(0~0~0)$ & $(0~0~1)$\\ \hline
$c_4$ & $(0~0~0)$ & $(1~0~0)$ & $(0~0~0)$ & $(1~0~1)$ & $(0~0~0)$ & $(0~0~1)$\\ \hline
$c_5$ & $(0~0~0)$ & $(0~1~0)$ & $(0~0~0)$ & $(0~0~0)$ & $(0~0~0)$ & $(0~0~0)$\\ \hline
$c_6$ & $(0~0~0)$ & $(0~0~1)$ & $(0~0~0)$ & $(1~0~1)$ & $(0~0~0)$ & $(0~0~1)$\\ \hline
$c_7$ & $(0~0~0)$ & $(0~0~0)$ & $(0~0~0)$ & $(0~0~0)$ & $(1~0~1)$ & $(0~0~1)$\\ \hline
$c_8$ & $(0~0~0)$ & $(0~0~0)$ & $(0~0~0)$ & $(0~0~0)$ & $(0~0~0)$ & $(0~1~0)$\\ \hline
$c'$ &  $(0~0~0)$ & $(0~0~0)$ & $(0~0~0)$ & $(0~0~0)$ & $(0~0~0)$ & $(1~0~1)$ \\\hline
\end{tabular}
\end{center}
The action of $\phi\in\Gal(\FF_{121}|\FF_{11})$ on $\Homcr(G,M)$ only affects the element $c_7$ of this basis, sending it to $\phi^\star c_7=c_7+c'$. Its action on $\HH^1(G,M)$ is therefore trivial.

\begin{rk}
We now know that the action of $\Gal(\FF_{121}|\FF_{11})$ on $\HH^1(E,\F)$ is trivial. Therefore, we could have chosen a subcover of $C\nt$ by first computing a basis of $\HH^1(C,\mu_2)(\FF_{11})$ and then taking $n^{\text{th}}$ roots of the functions appearing in this basis. With the notations above, such a basis is given by $(D_1,f_1), (D_2,f_2), (D_4,f_4),(D_3+\bar D_5,f_3f_5)$. Set $k(C')=k(C)(z_1,z_2,z_4,t_3)$ where $z_i^2=f_i$ and $t_3^2=f_3f_5$. Using the previous notations $h_3, h_4$, a preimage of $\sigma\in\Aut(C|E)$ in $\Aut(C'|E)$ is the automorphism $\delta$ given by: 
\[ 
x\mapsto -x, \quad y\mapsto y,\quad z_1\mapsto \frac{1}{z_1},\quad z_2\mapsto \frac{1}{z_2},\quad z_4\mapsto h_4(x,y)z_2z_4,\quad t_3\mapsto a^{30}h_3(x,y)t_3.\]
The group $\Aut(C'|E)$ only has order 32, and one can check that the map \[\RG(\Aut(C'|E),M)\to\RG(\Aut(C\nt|E),M)\] is a quasi-isomorphism.
\end{rk}

\begin{ack}
The author would like to thank David Madore and Fabrice Orgogozo for their continued help and support during his PhD thesis, whose main results led to the ones presented in this article. He is also grateful to the anonymous referee for their helpful questions and comments.
\end{ack}
\bigskip 

\let\oldaddcontentsline\addcontentsline
\renewcommand{\addcontentsline}[3]{}
\bibliography{article_computing_constructible}
\bibliographystyle{plain}
\end{document}